\documentclass[11pt,letterpaper]{article}
\usepackage[T1]{fontenc}
\usepackage{amssymb,amsmath,amsfonts,amsthm}
\usepackage{xcolor,mathtools,mleftright}

\usepackage{listings,longtable,array}
\usepackage[pdftex]{graphicx} 
\usepackage{pgf,tikz,pgfplots}
\pgfplotsset{compat=1.18}
\pgfkeys{/pgf/number format/.cd,1000 sep={}}
\usepackage[utf8]{inputenc}
\usetikzlibrary{arrows,automata}
\usetikzlibrary{positioning}

\usepackage[pdftex,linkcolor=black,pdfborder={0 0 0}]{hyperref} 
\usepackage{calc} 
\usepackage[shortlabels]{enumitem} 

\usepackage{comment}

\usepackage[margin=1in]{geometry} 

\usepackage{lipsum} 

\newtheorem{thm}{Theorem}[section]
\newtheorem{lem}[thm]{Lemma}

\newtheorem{prop}[thm]{Proposition}
\newtheorem{cor}[thm]{Corollary}
\newtheorem{conj}[thm]{Conjecture}

\theoremstyle{definition}

\newtheorem{alg}[thm]{Algorithm}

\theoremstyle{remark}
\newtheorem{rem}[thm]{Remark}

\newtheorem{exercise}[thm]{Exercise}

\DeclareMathOperator{\disc}{disc}

\DeclareMathOperator{\Stab}{Stab}
\DeclareMathOperator{\Cl}{Cl}

\newcommand{\GL}{\mathrm{GL}}
\newcommand{\SL}{\mathrm{SL}}
\newcommand{\PGL}{\mathrm{PGL}}

\newcommand{\QQ}{\mathbb{Q}}
\newcommand{\PP}{\mathbb{P}}
\newcommand{\RR}{\mathbb{R}}
\newcommand{\ZZ}{\mathbb{Z}}

\renewcommand{\aa}{\mathfrak{a}}
\newcommand{\pp}{\mathfrak{p}}

\newcommand{\F}{\mathcal{F}}

\renewcommand{\L}{\mathcal{L}}
\newcommand{\M}{\mathcal{M}}
\newcommand{\OO}{\mathcal{O}}
\renewcommand{\P}{\mathcal{P}}
\renewcommand{\S}{\mathcal{S}}
\newcommand{\T}{\mathcal{T}}
\newcommand{\W}{\mathcal{W}}
\newcommand{\Z}{\mathcal{Z}}

\newcommand{\RB}{\mathcal{RB}}
\newcommand{\DB}{\mathcal{DB}}
\newcommand{\SRT}{\mathcal{SRT}}

\newcommand{\ba}{\overline}
\newcommand{\bs}{\backslash}
\newcommand{\cross}{\times}

\newcommand{\textand}{\quad \text{and} \quad}
\newcommand{\textor}{\quad \text{or} \quad}
\renewcommand{\to}{\mathop{\rightarrow}\limits}
\newcommand{\size}[1]{\lvert #1 \rvert}
\let\left\mleft
\let\right\mright
\newcommand{\Size}[1]{\left\lvert #1 \right\rvert}
\newcommand{\floor}[1]{\left\lfloor #1 \right\rfloor}

\newcommand{\intsec}{\cap}
\newcommand{\union}{\cup}
\newcommand{\bigintsec}{\bigcap}
\newcommand{\isom}{\cong}
\newcommand{\<}{\left\langle}
\renewcommand{\>}{\right\rangle}
\newcommand{\ignore}[1]{}

\title{Raney Transducers\\and the Lowest Point of the $p$-Lagrange spectrum}
\author{CMU REU Spring 2024\\
  Brandon Dong, Soren Dupont, Evan M.\ O'Dorney, and W.\ Theo Waitkus 
}

\begin{document}
  
  \maketitle
  
  \begin{abstract}
    It is well known that the golden ratio $\phi$ is the ``most irrational'' number in the sense that its best rational approximations $s/t$ have error $\sim 1/(\sqrt{5} t^2)$ and this constant $\sqrt{5}$ is as low as possible. Given a prime $p$, how can we characterize the reals $x$ such that $x$ and $p x$ are both ``very irrational''?  This is tantamount to finding the lowest point of the \emph{$p$-Lagrange spectrum} $\mathcal{L}_p$ as previously defined by the third author. We describe an algorithm using Raney transducers that computes $\min \mathcal{L}_p$ if it terminates, which we conjecture it always does. We verify that $\min \mathcal{L}_p$ is the square root of a rational number for primes $p < 2000$. Mysteriously, the highest values of $\min \mathcal{L}_p$ occur for the Heegner primes $67$, $3$, and $163$, and for all $p$, the continued fractions of the corresponding very irrational numbers $x$ and $p x$ are in one of three symmetric relations.
  \end{abstract}
  
  \section{Introduction}
  We recall the classical notions of the Lagrange and Markoff spectra. If $\xi$ is an irrational real number, we define its \emph{Lagrange approximability}
  \begin{equation} \label{eq:lambda}
    \lambda(\xi) = \limsup_{\tfrac{s}{t} \in \QQ, \tfrac{s}{t} \to \xi} \frac{1}{t^2 \Size{\dfrac{s}{t} - \xi}} \quad \in \quad \RR \union \infty,
  \end{equation}
  and the \emph{Lagrange spectrum} $\L \subseteq \RR$ to be the set of real values attained by $\lambda(\cdot)$. One thinks of $\lambda(\xi)$ as the ease of approximation of $\xi$ by rationals. A classical result usually called Hurwitz's theorem (though it can justly be attributed to Markoff) states that $\lambda(\xi) \geq \sqrt{5}$, with equality when $\xi = \phi$ is the golden ratio, or one of its images under $\SL_2\ZZ$ (which are hence the ``most irrational'' numbers). Similarly, if $\xi \neq \xi'$ are two irrationals, then we define their \emph{Markoff approximability}
  \begin{equation} \label{eq:mu}
    \mu(\xi,\xi') = \sup_{(s,t) \in \ZZ^2 \bs \{0\}} \frac{\Size{\xi - \xi'}}{\Size{s - t\xi} \Size{s - t\xi'}} = \sup_{(s,t) \in \ZZ^2 \bs \{0\}} \frac{\sqrt{\disc f}}{\size{f(s,t)}} \quad \in \quad \RR \union \infty
  \end{equation}
  where $f(x,y) = a(x - y\xi)(x - y\xi')$ is the quadratic form with roots $\xi,\xi'$ (the scaling $a$ is arbitrary; it is often desirable for $f$ to have integer coefficients). The \emph{Markoff spectrum} $\M$ is then the set of real values of $\mu(\cdot,\cdot)$. This spectrum was first considered by Markoff \cite{MarkoffI} in the guise of infima of quadratic forms. The Lagrange and Markoff spectra each consist of an initial discrete segment below $3$, a mysterious fractal middle region which remains the topic of current research \cite{MoreiraGeometric,MoreiraMLIsNotClosed}, and \emph{Hall's ray} $[F,\infty)$ where $F = 4.5278\ldots$ was computed exactly by Fre\u{\i}man \cite{Freiman1975}, who also showed that $\L \subsetneq \M$ \cite{Freiman1968}. See \cite{CusickFlahive} for a comprehensive account of results up to 1989, while new connections continue to be unearthed (see \cite{AbeTwoColor}).
  
  In this paper, we fix a positive integer $n$ and introduce ``$n$-analogues'' of the Lagrange and Markoff spectra by inserting a factor $\gcd(t,n)$ in the numerators of the definitions of approximability \eqref{eq:lambda}--\eqref{eq:mu}, thus:
  \begin{align} 
    \lambda_n(\xi) &= \limsup_{\tfrac{s}{t} \in \QQ, \tfrac{s}{t} \to \xi} \frac{\gcd(t, n)}{t^2 \Size{\dfrac{s}{t} - \xi}}, & \L_n &= \{\lambda_n(\xi) \in \RR : \xi \in \RR\} \label{eq:lambda_n} \\
    \mu_n(\xi,\xi') &= \sup_{(s,t) \in \ZZ^2 \bs \{0\}} \frac{\gcd(t, n) \cdot \Size{\xi - \xi'}}{\Size{s - t\xi} \Size{s - t\xi'}}, & \M_n &= \{\lambda_n(\xi,\xi') \in \RR : \xi\neq \xi' \in \RR\} \label{eq:mu_n}
  \end{align}
  These notations were introduced by the third author in \cite{OConics} and shown to govern the intrinsic approximation of points on conics, generalizing previous work on the unit circle, which gives $\L_2$ \cite{KopetzkyApprox,KimSim,ChaIntrinsic}, and on the conic $x^2 + xy + y^2 = 1$, which gives $\L_3$ \cite{ChaEis}. The spectra $\M_n$ appear in the work of Schmidt \cite[p.~15]{SchMinII} by generalizing the work of Markoff on infima of binary quadratic forms.
  
  Classically, the approximabilities $\lambda(\xi)$, $\mu(\xi,\xi')$ are invariant under the action of $\SL_2\ZZ$ (indeed $\GL_2\ZZ$) by linear fractional transformations on the real projective line. Likewise, it is easy to see that $\lambda_n(\xi)$ and $\mu_n(\xi,\xi')$ are invariant under transformations by the congruence subgroup
  \[
    \Gamma^0(n) = \left\{\begin{bmatrix}
      a & b \\ c & d
    \end{bmatrix} \in \SL_2\ZZ : n \mid b\right\}
  \]
  (an $\SL_2\ZZ$-conjugate of the more familiar $\Gamma_0(n)$, where the divisibility condition is imposed on $c$ instead). More generally, Vulakh \cite{VulakhMarkov} defines a notion of Markoff spectrum for any Fuchsian subgroup of $\SL_2\RR$.
  
  It is natural to ask if various facts about $\L$ and $\M$ carry over to $\L_n$ and $\M_n$, and whether there are idiosyncratic behaviors for certain values of $n$. In this paper, we begin to answer these questions as regards the bottom of the spectrum, where we expect to find a countable discrete sequence of isolated points converging to the first limit point. We restrict to $n = p$ prime (a simplification, as the gcd in \eqref{eq:lambda_n}--\eqref{eq:mu_n} can then only take two values). Finding $\min \L_p$ can be thought of as computing a $\xi$ such that $\xi$ and $p\xi$ are both ``very irrational,'' that is, hard to approximate by rationals of small denominator.
  
  Schmidt \cite{SchMinI,SchMinII} computes the initial discrete segment of $\M_n$ for $n = 2,3,5,6$. Vulakh \cite{VulakhMarkov} does the same for $n = 13$ (Theorem 32) and claims that ``[t]he results obtained in the preceding sections can be used to find the discrete part of $\M_m$ in those cases [not already solved by Schmidt]'' \cite[p.~4090]{VulakhMarkov}. However, no general algorithm is given, nor is a general theorem stated on the structure of $\M_n$ for all $n$. In this paper, we begin to fill this gap. We show that $\L_n$ and $\M_n$ have a common minimum, and we give an algorithm that, if it terminates, computes $\min \L_p$ for a prime $p$ and verifies the following:
  \begin{conj}\label{conj:main}
    For all positive integers $n$, the lowest point $\min \L_n$ is an isolated point of $\L_n$ and is the irrational square root of a rational number.
  \end{conj}
  \begin{thm} \label{thm:main}
    Conjecture \ref{conj:main} is verified for all prime values $n = p < 2000$.
  \end{thm}
  The lowest point $\min \L_p$ fluctuates with $p$ (see Section \ref{sec:data}) and reaches its highest observed value at $\min \L_{67} = 3.678\ldots,$ leading to the following curiosity:
  \begin{thm} \label{thm:67}
    The prime $p = 67$ is the unique prime $p < 2000$ with the following property: There is no irrational $\xi$ such that the continued fraction expansions of $\xi$ and $p\xi$ consist after a certain point of only $1$\kern-2pt's and $2$\kern-2pt's.
  \end{thm}

  Along the way, we prove that $\L_n$ and $\M_n$ enjoy certain desirable properties long known for $\L$ and $\M$, specifically:
  \begin{itemize}
    \item $\L_n \subseteq \M_n$ (Proposition \ref{prop:LnMn});
    \item Any $\alpha \in \M_n$ is realized by a pair $(\xi,\xi')$ for which the supremum in the definition \eqref{eq:mu_n} is attained (Proposition \ref{prop:attained});
    \item $\L_n$ and $\M_n$ are closed (Propositions \ref{prop:Ln_closed} and \ref{prop:Mn_closed} respectively);
    \item $\L_n$ and $\M_n$ contain a Hall's ray $[nF, \infty)$ (Proposition \ref{prop:hall}).
  \end{itemize}

  \subsection{Methods}
  Kim--Sim \cite{KimSim} study and compare $\L_2$ and $\M_2$ using \emph{Romik expansions,} a useful way of simultaneously recording the continued fraction expansions of $\xi$ and $2\xi$ by words over an alphabet of three digits. Cha--Chapman--Gelb--Weiss \cite{ChaEis} create a suitable analogue of Romik expansions with a five-digit alphabet useful for studying $\L_3$ and $\M_3$. However, for reasons that we will explain below, we expect no analogous code with a finite alphabet to exist for $n > 3$. Hence we need another approach. It is unclear if the geometrical methods of Vulakh can be harnessed for automated computation. Instead, we use \emph{Raney transducers,} a technique for applying a linear fractional transformation to a real number expressed in continued fraction form \cite{Raney_Automata}. A Raney transducer is a finite directed graph whose edges are labeled with segments of a continued fraction. Our algorithm iterates through paths on the Raney transducer with an eye to looking for suitable cycles, representing periodic continued fractions with the desired low approximabilities.

  \subsection{Organization of the paper}
  In Section \ref{sec:classical}, we recall classical results on the spectra $\L$ and $\M$. In Section \ref{sec:basic}, we prove some elementary results on $\L_n$ and $\M_n$ by working from the definitions. In Section \ref{sec:raney}, we describe the construction of two types of Raney transducer: a ``fast'' one, due to Raney, that is useful for computations, and a ``slow'' one that has certain theoretical advantages. In Section \ref{sec:closed}, we apply Raney transducers to show that $\L_n$ is closed. In Section \ref{sec:alg}, we describe an algorithm for computing $\min\L_p$ for prime $p$ and verifying Conjecture \ref{conj:main}. In Section \ref{sec:Markoff}, we explain certain low-lying points connected with Markoff triples that appear repeatedly as the output of our algorithm. Finally, in Section \ref{sec:data}, we display data and muse on patterns found in the outputs for all $p$.

  \subsection{Code}
  The Sage code used in the computational parts of the paper can be found at \url{https://github.com/sad-ish-cat/DioApprox}. In the comments at the end of the file \texttt{raney.sage} are some sample commands as a guide to replicating the computations.

  \subsection{Acknowledgements}
  This paper is an outgrowth of our results from the course ``Topics in Undergraduate Research'' at Carnegie Mellon University, spring 2024. We thank Theresa Anderson for organizing the course.

  \section{Classical results on approximability} \label{sec:classical}

  If $\xi$ is an irrational number, we define the \emph{quality} of an approximation $s/t$ to be the quantity
  \[
    \frac{1}{t^2 \Size{\dfrac{s}{t} - \xi}}
  \]
  appearing in \eqref{eq:lambda}. The following results are classical.
  \begin{prop} \label{prop:quality}
    Let $\xi = [a_0, a_1, a_2, \ldots]$ be an irrational number, expressed as an infinite simple continued fraction.
    \begin{enumerate}[(a)]
      \item The quality of the convergent $p_k/q_k = [a_0, \ldots, a_{k-1}]$ is given by
      \[
      \frac{(-1)^k}{q_k^2\bigl(\frac{p_k}{q_k} - \xi\bigr)} = a_k + [0,a_{k-1},a_{k-2}, \ldots, a_1] + [0,a_{k+1},a_{k+2}, \ldots]\big).
      \]
      \item Any approximation to $\xi$ of quality $\geq 2$ is a convergent.
      \item The approximability of $\xi$ is the limsup of the qualities of its convergents:
      \begin{equation}\label{eq:approx}
        \lambda(\xi) = \limsup_{k \to \infty} \big(a_k + [0,a_{k-1},a_{k-2}, \ldots, a_1] + [0,a_{k+1},a_{k+2}, \ldots]\big).
      \end{equation}
    \end{enumerate}
  \end{prop}
  \begin{proof}
    These are standard results; see \cite[Appendix 1]{CusickFlahive}.
  \end{proof}

  A \emph{cut} of a continued fraction is a choice of truncation point $k$ as above. A notation like $[a_0, \ldots, a_{k-1} \big\vert a_k, \ldots]$ is often used; but this obscures the symmetry between the terms before and after $a_k$. For this reason, we denote a cut by $[a_0, \ldots, a_{k-1} \boxed{a_k} a_{k+1}, \ldots]$. The \emph{quality} of the cut is the $\lambda$-value
  \[
    \lambda\bigl([a_0, \ldots, a_{k-1} \boxed{a_k} a_{k+1}, \ldots]\bigr) =
    a_k + [0,a_{k-1},a_{k-2}, \ldots, a_1] + [0,a_{k+1},a_{k+2}, \ldots].
  \]
  Note that the dominant contribution to the quality $\lambda(C)$ of a cut is $a_k$; the terms $a_{k \pm h}$ lying farther away are progressively less important as $h$ grows. This can be made quantitative:
  \begin{lem} \label{lem:sensitive}
    If two simple continued fractions $\xi = [a_0,a_1,\ldots]$ and $\xi' = [a_0',a_1',\ldots]$ have the same initial terms $[a_0,\ldots, a_k] = [a_0',\ldots, a_k']$, then the difference of their values is bounded by
    \[
      \size{\xi - \xi'} \leq 2^{1-k}.
    \]
  \end{lem}
  \begin{proof}
    This is standard and elementary; see \cite[Lemma 1]{CusickFlahive}.
  \end{proof}

  \subsection{Approximabilities of quadratic irrationals} \label{sec:quad_approx}
  To compute $\lambda(\xi)$ for a quadratic irrational $\xi$ is a finite task: if
  \[
    \xi = [a_0,\ldots, a_k, \overline{b_1,\ldots, b_\ell}],
  \]
  then the transient terms $a_i$ are of no consequence, and taking the limit of the qualities of cuts at a recurring term $b_i$, we may write
  \begin{equation} \label{eq:cut_periodic}
    \lambda\bigl(\overline{b_1,\ldots\boxed{b_i}\ldots,b_\ell}\bigr) = 
    b_i + [0, \overline{b_{i+1}, \ldots, b_\ell, b_1, \ldots, b_i}]
        + [0, \overline{b_{i-1}, \ldots, b_1, b_\ell, \ldots, b_i}],
  \end{equation}
  so that
  \begin{equation}
    \lambda(\xi) = \max_{1 \leq i \leq \ell} \bigl(\overline{b_1,\ldots\boxed{b_i}\ldots,b_\ell}\bigr).
    \label{eq:periodic}
  \end{equation}

  If $\bar\xi$ is the algebraic conjugate of $\xi$, note that $\mu(\xi,\bar\xi) = \lambda(\xi)$. The formula for the conjugate of a purely periodic continued fraction,
  \[
    \operatorname{conj}([\overline{b_1,\ldots,b_\ell}]) = -[0, \overline{b_\ell,\ldots, b_1}],
  \]
  implies that the right-hand side of \eqref{eq:cut_periodic} and hence \eqref{eq:periodic} is the difference of a quadratic irrational and its conjugate, and hence is the irrational square root of a rational number.
  
  \section{Elementary properties of \texorpdfstring{$\L_n$}{ℒₙ} and \texorpdfstring{$\M_n$}{ℳₙ}} \label{sec:basic}
  
  We begin by proving some properties of general interest about the $n$-spectra $\L_n$ and $\M_n$. In the classical case $n = 1$, these were mostly proved in Cusick \cite{CusickConnection} and/or Cusick--Flahive \cite{CusickFlahive}. First, we show that the $n$-approximability $\lambda_n(\xi), \mu_n(\xi,\xi')$ defined in \eqref{eq:lambda_n}--\eqref{eq:mu_n} can be computed in terms of the classical approximability $\lambda(\xi), \mu(\xi,\xi')$ respectively.
  
  \begin{prop} \label{prop:lambda_n->lambda}
    \[
    \lambda_n(\xi) = \max_{g\mid n} \lambda(g\xi) \textand 
    \mu_n(\xi,\xi') = \max_{g\mid n} \mu(g\xi, g\xi').
    \]
  \end{prop}
  \begin{proof}
    We have
    \begin{align*}
      \lambda_n(\xi) &= \limsup_{\tfrac{s}{t} \to \xi} \frac{\gcd(t,n)}{t^2 \Size{\tfrac{s}{t} - \xi}} \\
      &= \limsup_{\tfrac{s}{t} \to \xi} \max_{g\mid n, g\mid t} \frac{g}{t^2 \Size{\tfrac{s}{t} - \xi}} \\
      &= \max_{g \mid n} \limsup_{\tfrac{s}{t} \to \xi, g\mid t} \frac{g}{t^2 \Size{\tfrac{s}{t} - \xi}} \\
      &= \max_{g \mid n} \limsup_{\tfrac{s}{t'} \to g\xi} \frac{g}{(gt')^2 \bigl\lvert{\tfrac{s}{gt'} - \xi}\bigr\rvert} \\
      &= \max_{g \mid n} \limsup_{\tfrac{s}{t'} \to g\xi} \frac{1}{t'^2 \Size{\tfrac{s}{t'} - g\xi}} \\
      &= \max_{g \mid n} \lambda(g\xi).
    \end{align*}
    Note that in the transformation process, we must allow non-reduced fractions $s/t$ and $s/t'$, but including or excluding such fractions does not affect the supremum because the reduced form always yields a better approximation quality. The proof for $\mu_n$ is analogous.
  \end{proof}
  \begin{cor}
    If $m \mid n$, then
    \begin{equation}\label{eq:n'|n}
      \L_n \subseteq \L_m \textand \M_n \subseteq \M_m.
    \end{equation}
    In particular,
    \begin{equation} \label{eq:LnL}
      \L_n \subseteq \L \textand \M_n \subseteq \M.
    \end{equation}
  \end{cor}
  \begin{rem}
    We take this corollary as a sign that we have chosen the ``correct'' scaling of $\L_n$ and $\M_n$. In the literature \cite{ChaIntrinsic,ChaEis,KimSim}, these spectra have arisen in different contexts and are scaled by a multiplicative constant. With some authors, such as Vulakh \cite{VulakhMarkov}, the approximabilities are the reciprocals of those given here.
  \end{rem}
  
  \begin{prop} \label{prop:LnMn}
    $\L_n \subseteq \M_n$. Moreover, any $\alpha \in \L_n$ can be realized as $\mu_n(\xi,\xi')$ for some $\xi,\xi' \in \RR$ for which the supremum in the definition \eqref{eq:mu_n} is attained at some $(s,t) \in \ZZ^2 \bs \{0\}$.
  \end{prop}
  
  \begin{proof}
    Let $\alpha = \lambda_n(\xi) \in \L_n$. Then there is a sequence of fractions
    \[
    \frac{s_1}{t_1}, \frac{s_2}{t_2}, \ldots
    \]
    converging to $\xi$ such that the ``$n$-qualities,'' that is, the arguments to the limsup in the definition \eqref{eq:lambda_n}
    \[
    \alpha_i \coloneqq \frac{\gcd(t_i,n)}{t_i^2 \bigl\lvert \frac{s_i}{t_i} - \xi \bigr\rvert},
    \]
    converge to $\alpha$. The action of $\Gamma^0(n)$ on $\PP^1(\QQ)$ has finitely many orbits (the \emph{cusps} of the associated modular curve $X_0(n)$). Hence, after passing to a subsequence, we may assume that each $s_i/t_i$ maps by some $\gamma_i \in \Gamma^0(n)$ to a single fraction $s/t$. Note that $\gcd(t_i,n) = \gcd(t,n)$. The map $\gamma_i$ is unique up to postcomposition by the stabilizer $\Z = \Stab_{\Gamma^0(n)}(s/t)$, which is isomorphic to $\ZZ$, generated by one parabolic element with unique fixed point $s/t$. Thus, fixing a closed fundamental domain $\F$ for $\Z$, not containing $s/t$ nor (for simplicity) $\infty$, we can pick $\gamma_i$ such that $\gamma_i(\infty) \in \F$. Now, after passing to a subsequence, we can find $\xi', \xi'' \in \PP^1(\RR)$ such that
    \[
    \gamma_i(\infty) \to \xi', \quad \gamma_i(\xi) \to \xi''.
    \]
    We claim that $(\xi',\xi'')$ is the desired pair with $\mu_n(\xi',\xi'') = \alpha$. We have
    \begin{equation} \label{eq:LnMn_inv}
      \alpha = \lim_{i \to \infty} \frac{\gcd(t_i,n)}{t_i^2 \bigl\lvert \frac{s_i}{t_i} - \xi \bigr\rvert}
      = \lim_{i \to \infty} \frac{\gcd(t,n) \size{\gamma_i(\xi) - \gamma_i(\infty)}}{\size{s - t\gamma_i(\xi)} \size{s - t\gamma_i(\infty)}}.
    \end{equation}
    First note that if $\xi' = \xi''$, then as $i \to \infty$, the numerator of \eqref{eq:LnMn_inv} tends to $0$ while the denominator is bounded since $s/t, \infty \notin \F$. So $\alpha = 0$, which is impossible. So $\xi' \neq \xi''$.
    
    For any reduced fraction $u/v$, let $u_i/v_i = \gamma_i^{-1}(u/v)$ and note that, by the $\Gamma^0(n)$-invariance of the Markoff $n$-approximability,
    \begin{align}
      \frac{\gcd(v,n) \Size{\xi' - \xi''}}{\Size{u - v\xi'}\Size{u - v\xi''}} &= \lim_{i \to \infty} \frac{\gcd(v,n) \Size{\gamma_i(\infty) - \gamma_i(\xi)}}{\Size{u - v\gamma_i(\infty)}\Size{u - v\gamma_i(\xi)}} \nonumber \\
      &= \lim_{i \to \infty} \frac{\gcd(v_i, n)}{v_i \Size{u_i - v_i \xi}}. \label{x:uv}
    \end{align}
    We claim that $u_i/v_i \to \xi$, which will complete the proof, since the right-hand side of \eqref{x:uv} is bounded above by $\alpha$, with equality holding when $u/v = s/t$.
    
    Suppose not. After passing to a subsequence, $u_i/v_i$ converge to some point $\eta \in \PP^1(\RR)$ different from $\xi$. Now
    \begin{equation} \label{x:toxi}
      \frac{\gcd(v_i,n)}{v_i \size{u_i - \xi v_i}} = \frac{\gcd(v,n) \size{\gamma_i(\xi) - \gamma_i(\infty)}}{\size{u - \gamma_i(\xi) v}\size{u - \gamma_i(\infty) v}} \to \frac{\gcd(v,n) \size{\xi' - \xi''}}{\size{u - \xi' v} \size{u - \xi'' v}} > 0.
    \end{equation}
    Since $\eta \neq \xi$, the left-hand side of \eqref{x:toxi} goes to $0$ unless $u_i/v_i$ is infinitely often equal to the same fraction. Passing to a subsequence, we assume that $u_i/v_i = \eta = u_1/v_1$ is constant. Now $\gamma_i = \sigma^{m_i} \gamma_1$ lies in a fixed coset of the stabilizer $\Z_{\eta} = \sigma^\ZZ$ of $\eta$ in $\Gamma^0(n)$. Since the $\gamma_i$ cannot be constant on an infinite subsequence, we must have $\size{m_i} \to \infty$ and so, for any $x \in \PP^1(\RR)$ different from $u/v$, we have $\lim_{i \to \infty} \gamma_i(x) = \eta$. In particular, $\xi' = \eta = \xi''$, which is a contradiction.
  \end{proof}
  
  \begin{prop} \label{prop:attained}
    Any $\alpha \in \M_n$ can be realized by some $\xi,\xi' \in \RR$ for which the supremum in the definition \eqref{eq:mu_n} is attained at some $(s,t) \in \ZZ^2 \bs \{0\}$.
  \end{prop}
  
  \begin{proof}
    Let $\alpha = \mu_n(\xi,\xi')$. If the supremum $\alpha$ in \eqref{eq:mu_n} is not attained, then there is an infinite sequence of distinct pairs $(s_i,t_i)$ for which
    \begin{equation} \label{x:attained}
      \lim_{i \to \infty} \frac{\gcd(t, n) \cdot \Size{\xi - \xi'}}{\Size{s_i - t_i\xi} \Size{s_i - t_i\xi'}} = \alpha.
    \end{equation}
    Since $\size{s_i} + \size{t_i} \to \infty$, the two factors in the denominator cannot both be bounded, so on passing to a subsequence, one of them tends to $\infty$ and the other to $0$. WLOG $s_i/t_i \to \xi$. Then
    \[
    \alpha = \limsup_{\tfrac{s}{t} \to \xi} \frac{\gcd(t, n) \cdot \Size{\xi - \xi'}}{\Size{s - t\xi} \Size{s - t\xi'}} = \limsup_{\tfrac{s}{t} \to \xi} \frac{\gcd(t, n)}{t \Size{s - t\xi}} = \lambda_n(\xi).
    \]
    Hence $\alpha \in \L_n$. By the previous proposition, $\alpha \in \M_n$ is realized by a $(\xi'',\xi''')$ for which the supremum is attained.
  \end{proof}
  
  \begin{prop} \label{prop:Mn_closed}
    $\M_n$ is closed.
  \end{prop}
  \begin{proof}
    Let $\alpha_1, \alpha_2, \ldots$ be a sequence of elements in $\M_n$ tending to a limit $\alpha$; we shall show that $\alpha \in \M$. By Proposition \ref{prop:attained}, each $\alpha_i = \mu_n(\xi_i, \xi'_i)$ with the supremum being attained at some $s_i/t_i \in \PP^1(\QQ)$. Applying $\Gamma^0(n)$, we can transform each $s_i/t_i$ to one of finitely many values (the cusps of $X_0(n)$, as above), and then, passing to a subsequence, we may assume that $s_i/t_i = s/t$ are all equal. Let $\Z \isom \ZZ$ be the stabilizer of $s/t$ and $\F$ be a fundamental domain for $\Z$ as in the proof of Proposition \ref{prop:LnMn}. Applying elements of $\Z$, we may assume that $\xi_i \in \F$. Passing to a subsequence again, we may assume that $\xi_i \to \xi$ and $\xi_i' \to \xi'$ converge in $\PP^1(\RR)$. For any $(u,v) \in \ZZ^2\bs \{0\}$,
    \begin{align}
      \frac{\gcd(v,n) \Size{\xi - \xi'}}{\Size{u - v\xi}\Size{u - v\xi'}} &= \lim_{i \to \infty} \frac{\gcd(v,n) \Size{\xi_i - \xi_i'}} {\Size{u - v\xi_i}\Size{u - v\xi_i'}} \label{x:uvx} \\
      &\leq \lim_{i \to \infty} \alpha_i \nonumber \\
      &= \alpha, \nonumber
    \end{align}
    equality holding when $u/v = s/t$. In particular, $\xi \neq \xi'$ because $\alpha \neq 0$ and $\xi \neq s/t$, as in the proof of Proposition \ref{prop:LnMn}. Also, $\xi$ and $\xi'$ are finite and irrational because otherwise there would be a choice of $(u,v)$ for which the left-hand side of \eqref{x:uvx} tends to infinity. So $\mu_n(\xi, \xi') = \alpha$, as desired.
  \end{proof}
  
  The $n$-Lagrange spectrum $\L_n$ is also closed, but the proof involves Raney transducers and thus will be taken up in the next two sections. Here are a few other elementary facts.
  
  \begin{prop} \label{prop:min}
    $\min \L_n = \min \M_n$.
  \end{prop}
  \begin{proof}
    Since $\M_n$ is closed and bounded below, it has a minimal element $\alpha = \mu_n(\xi,\xi')$. We have $\mu_n(\xi,\xi') \geq \lambda_n(\xi)$. But $\lambda_n(\xi) \in \L_n \subseteq \M_n$, so equality holds and $\alpha = \lambda_n(\xi) \in \L_n$.
  \end{proof}

  \begin{prop} \label{prop:scale}
    $n\L \subseteq \L_n$ and $n\M \subseteq \M_n$. Here $n\L = \{n\alpha : \alpha \in \L\}$ and likewise for $n\M$.
  \end{prop}
  \begin{proof}
    Let $\alpha = \lambda(\xi) \in \L$. Then $\alpha$ is the limit of the $n$-qualities of a sequence of approximations $s_1/t_1, s_2/t_2, \ldots$ tending to $\xi$. Passing to a subsequence, we may assume that all the $s_i \equiv s$ and all the $t_i \equiv t$ are congruent modulo $n$. Then, applying a transformation in $\GL_2 \ZZ$ to $\xi$, we may assume that $t = 0$. Then $\lambda_n(\xi) = n\alpha$, since no sequence of approximations can do better than $s_1/t_1, s_2/t_2, \ldots.$ This proves that $n\L \subseteq \L_n$.

    The proof that $n\M \subseteq \M_n$ is similar but even easier, since by Proposition \ref{prop:attained}, any $\alpha \in \M$ is achieved by a $(\xi,\xi')$ such that the quality
    \[
      \frac{\Size{\xi - \xi'}}{\Size{s - t\xi} \Size{s - t\xi'}}
    \]
    attains its maximum at some $(s,t)$. Applying a $\GL_2\ZZ$-transformation, we may assume that $(s,t) = (1,0)$. We then observe that $\mu_{n}(\xi,\xi') = n\alpha$.
  \end{proof}
  As an immediate corollary, we get a Hall's ray for the $n$-spectra.
  \begin{prop} \label{prop:hall}
    $[nF, \infty) \subseteq \L_n \subseteq \M_n$, where $F = 4.5278\ldots$ is Fre\u{\i}man's constant, the least $F$ such that $[F, \infty) \subseteq \L$.
  \end{prop}
  
  \section{Raney transducers} \label{sec:raney}

  \subsection{\texorpdfstring{$LR$}{LR}-sequences}
  $LR$-sequences are a beautiful and handy alternative way to think about continued fractions. They appear to have been discovered several times, going back to Hurwitz \cite[\textsection 5]{Hurwitz1894}, who used the signs $+$, $-$ instead of $R$, $L$ respectively. A pleasant exposition is given by Series \cite{SeriesIntelligencer}.
  
  Consider the linear fractional transformations (LFT's)
  \[
  L = \begin{bmatrix}
    1 & 0 \\
    1 & 1 
  \end{bmatrix}, \quad
  R = \begin{bmatrix}
    1 & 1 \\
    0 & 1 
  \end{bmatrix},
  \]
  that is,
  \[
  L(x) = \frac{x}{x+1}, \quad R(x) = x+1,
  \]
  which map the real interval $(0,\infty)$ to the subintervals $(0,1)$ and $(1, \infty)$ respectively. Given $\xi \in (0, \infty)$, repeatedly apply $L^{-1}$ or $R^{-1}$, as needed to keep the value positive, stopping if the value $1$ is reached. If $\xi$ is rational, this process yields a finite $LR$-expansion $R^{a_0} L^{a_1} R^{a_2} \cdots (L \text{ or } R)^{a_k}$ ($a_0 \geq 0$, all other $a_i \geq 1$), characterized by
  \[
  \xi = R^{a_0} L^{a_1} R^{a_2} \cdots (L \text{ or } R)^{a_k}(1) = [a_0, a_1, a_2, \ldots, a_{k-1}, a_k + 1].
  \]
  If $\xi$ is irrational, we instead get an infinite $LR$-expansion $R^{a_0} L^{a_1} R^{a_2} \cdots$, and
  \begin{equation} \label{eq:nested ranges}
    \{\xi\} = \bigintsec_{k \geq 0} R^{a_0} L^{a_1} R^{a_2} \cdots (L \text{ or } R)^{a_k}[0,\infty], \quad \xi = [a_0, a_1, a_2, \ldots].
  \end{equation}
  Conversely, any infinite $LR$-sequence represents a unique positive irrational, unless the sequence ends with a constant tail $L^\infty$ or $R^\infty$. For sequences with a constant tail, the intersection point as in \eqref{eq:nested ranges} is rational; each positive rational has two infinite $LR$-expansions formed by appending $LR^\infty$ or $RL^\infty$ to its canonical finite $LR$-expansion.
  
  Although we will not need it in this paper, we would be remiss to omit the following beautiful geometric interpretation of the $LR$-expansion. Given a positive real number $\xi$, consider the geodesic from $i$ to $\xi$ in the hyperbolic upper half plane (Figure \ref{figure:hyp_tiling1}). As it passes through the tessellation formed by applying $\SL_2\ZZ$ to the geodesic $(0,\infty)$, check whether it exits each successive triangle to the left $(L)$ or the right $(R)$. The resulting sequence is the $LR$-expansion of $\xi$, either finite or infinite.

  \newcommand{\drawArc}[2]{%
    \pgfmathparse{((#2) - (#1))/2} \let\radius\pgfmathresult
    \draw (#2,0) arc (0:180:\radius);
  }
  \newcommand{\periodicArc}[2]{%
    \drawArc{#1}{#2}
    \drawArc{#1 + 1}{#2 + 1}
  }
  \newcommand{\mirrorArc}[2]{%
    \drawArc{#1}{#2}
    \drawArc{1 - #2}{1 - #1}
  }
  \newcommand{\periodicMirrorArc}[2]{%
    \periodicArc{#1}{#2}
    \periodicArc{1 - #2}{1 - #1}
  }
  \begin{figure}[ht]
    \centering
    \begin{tikzpicture}[xscale=4,yscale=4]
      \def\m{4/3};
      \pgfmathparse{(\m + 1/(\m))/2} \let\radius\pgfmathresult
      \pgfmathparse{2*atan(\m)} \let\angle\pgfmathresult
      \draw[red] (\m,0) arc (0:\angle:\radius);
      \node[below] at (0,0) (0) {$0$};
      \node[below] at (1,0) (0) {$1$};
      \node[below] at (2,0) (0) {$2$};
      \node[below] at (\m,0) (0) {$\xi$};
      \node[above right] at (1,0.75) {$R$};
      \node[above right] at (1.2,0.45) {$L$};
      \node[above right] at (1.29,0.23) {$L$};
      \periodicArc{0}{1}
      \periodicMirrorArc{0/1}{1/2}
      \periodicMirrorArc{0/1}{1/3}
      \periodicMirrorArc{0/1}{1/4}
      \periodicMirrorArc{0/1}{1/5}
      \periodicMirrorArc{0/1}{1/6}
      \periodicMirrorArc{1/6}{1/5}
      \periodicMirrorArc{1/5}{1/4}
      \periodicMirrorArc{1/4}{1/3}
      \periodicMirrorArc{1/3}{1/2}
      \periodicMirrorArc{1/3}{2/5}
      \periodicMirrorArc{2/5}{1/2}
      \draw
      (0,0) -- (0,1.2)
      (1,0) -- (1,1.2)
      (0,0) -- (2,0) -- (2,1.2);
    \end{tikzpicture}
    \caption{The expansion $RLL$ of a rational number $\xi = 4/3$}
    \label{figure:hyp_tiling1}
  \end{figure}
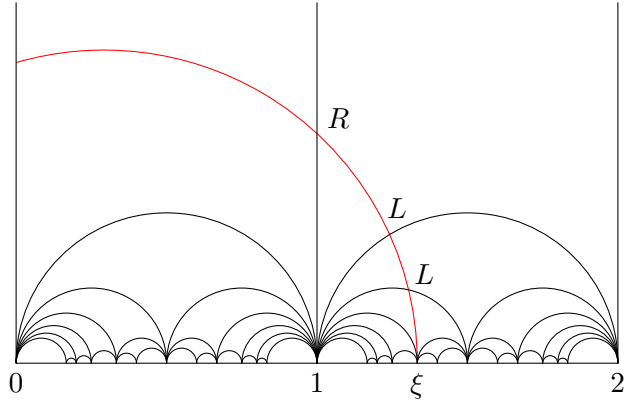
  
  \subsection{The slow Raney transducer}
  A \emph{Raney transducer} is a finite automaton, introduced by Raney in 1973 \cite{Raney_Automata}, associated to an LFT $\gamma \in \PGL_2(\QQ)$, that computes the $LR$-expansion of $\gamma(\xi)$ given that of $\xi$. Such transducers are useful in many questions related to continued fractions; see \cite{LiardetAlgebraic,RadaBounds}. See also Sol's expository thesis \cite{SolRaney}. In this paper, we are interested in the transformation $\gamma(\xi) = n\xi$.
  
  Here we construct a graph which we call the \emph{slow Raney transducer} because it computes the same transformations, albeit in a simpler and somewhat less efficient way. The naturalness of the construction will be useful for our proofs.
  
  Let
  \[
  \gamma(x) = \frac{a x + b}{c x + d}
  \]
  be an LFT. Assume that the coefficients $a,b,c,d$ are nonnegative integers and the determinant $n = ad - bc$ is positive, which implies that $\gamma$ maps $[0,\infty]$ into itself in an orientation-preserving way. Scale $a,b,c,d$ to be coprime nonnegative integers. Let $\xi \in [0,\infty]$ be a real number with $LR$-expansion $S_1S_2S_3\ldots$, where each letter $S_i \in \{L, R\}$. Suppose you want to compute the $LR$-expansion of $\gamma(\xi)$. When the first letter $S_1$ is revealed, it may or may not determine the first letter of $\gamma(\xi)$. The operative condition is whether
  \[
  \gamma\bigl(S_1([0,\infty])\bigr) \subseteq [0,1] \textor [1, \infty]
  \]
  If it is, then $\gamma(\xi)$ has an $LR$-expansion starting with $T_1 = L$ or $T_1 = R$ respectively, and the new remaining task is to compute the $LR$-expansion of
  \[
  \bigl(T_1^{-1} \circ \gamma \circ S_1\bigr) (\xi),
  \]
  where the parenthesized LFT again has nonnegative, coprime integer coefficients and determinant $n$. On the other hand, if $\gamma\bigl(S_1([0,\infty])\bigr)$ straddles the point $1$, then no letters of the $LR$-expansion of $\gamma(\xi)$ can be determined. This happens if $\gamma \circ S_1$ is given by a matrix
  \[
  M = \begin{bmatrix}
    a & b \\
    c & d
  \end{bmatrix}
  \]
  satisfying the inequalities
  \[
  a > c \textand b < d;
  \]
  in other words, $M$ is \emph{row-balanced} in Raney's terminology (i.e.\ it has no dominant row). We can reduce the computation of general rational LFT's to row-balanced ones. This leads to the following algorithm.
  
  \begin{alg}[Slow Raney Transducer]~ \label{alg:slow_Raney}
    \begin{itemize}
      \item[] \textbf{Input:} A row-balanced LFT $\gamma$, and the first letter $S$ of the $LR$-expansion of a positive real number $\xi = S(\xi')$.
      \item[] \textbf{Output:} A new row-balanced LFT $\gamma'$, and a word $W$ in the alphabet $\{L,R\}$ such that $\gamma(\xi) = W \gamma'(\xi')$.
      \item[] \textbf{Method:} \begin{enumerate}
        \item Start by trying $W_1 \leftarrow ()$ (the empty word) and $\gamma' = \gamma S^{-1}$.
        \item\label{it:mat} Write
        \[
        \gamma' = \begin{bmatrix}
          a & b \\
          c & d
        \end{bmatrix},
        \]
        with $a,b,c,d \geq 0$ and $ad - bc = n > 0$.
        \item If $a < c$ (implying $b < d$), set $W \leftarrow W L$, $\gamma' \leftarrow L^{-1} \gamma'$ and return to step \ref{it:mat}.
        \item If $b > d$ (implying $a > c$), set $W \leftarrow W R$, $\gamma' \leftarrow R^{-1} \gamma'$ and return to step \ref{it:mat}.
        \item Otherwise, $\gamma'$ is row-balanced. Return $W$ and $\gamma'$.
      \end{enumerate}
    \end{itemize}
  \end{alg}
  For fixed $n$, there is a finite set $\RB_n$ of row-balanced matrices, also called \emph{orphans} by Nathanson \cite{Nathanson_A_forest}, who produced a theory of continued fraction expansions of LFT's independently from Raney's. The number of orphans of each level $n$ has been tabulated in the OEIS \cite{orphan}. Hence we can encapsulate the results of Algorithm \ref{alg:slow_Raney} in a finite directed graph $\SRT_n$ whose nodes are $\RB_n$ and each node $\gamma$ has two outgoing edges
  \[
  \gamma \xrightarrow{S:W} \gamma'
  \]
  describing the output $(W, \gamma)$ of Algorithm \ref{alg:slow_Raney} to the input $(\gamma,S)$ for $S \in \{L, R\}$. We call this graph the \emph{slow Raney transducer of level $n$}. Some representative examples are shown in Figures \ref{fig:slow_Raney_2}--\ref{fig:slow_Raney_3}.

  \begin{figure}
    \centering
    \begin{tikzpicture}[->,>=stealth',auto,node distance=1.5in]
      \tikzstyle{every state}=[rectangle,inner sep=0]
      \node[state] (0) 
      {$\begin{matrix}
          2 & 0 \\ 0 & 1
        \end{matrix}$};
      \node[state] (0a) [right of=0]
      {$\begin{matrix}
          2 & 0 \\ 1 & 1
        \end{matrix}$};
      \node[state] (1) [below of=0a]
      {$\begin{matrix}
          1 & 0 \\ 0 & 2
        \end{matrix}$};
      \node[state] (1a) [below of=0]
      {$\begin{matrix}
          1 & 1 \\ 0 & 2
        \end{matrix}$};
      
      \path (0) edge [bend left] node {$L:-$ } (0a)
      (0a) edge [bend left] node {$L:L$ } (0)
      (0) edge [loop above] node {$R:RR$ } (0)
      (0a) edge [bend left]  node {$R:RL$} (1)
      (1) edge [bend left] node {$R:-$ } (1a)
      (1a) edge [bend left] node {$R:R$ } (1)
      (1) edge [loop below] node {$L:LL$ } (1)
      (1a) edge [bend left]  node {$L:LR$} (0);
    \end{tikzpicture}
    \caption{The slow Raney transducer $\SRT_2$.}
    \label{fig:slow_Raney_2}
  \end{figure}
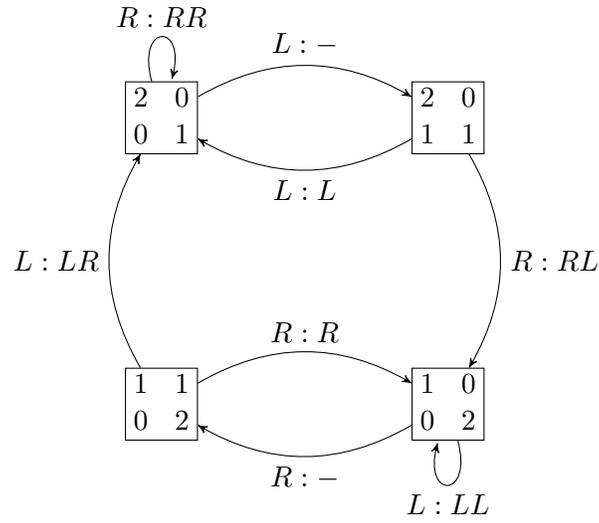
  
  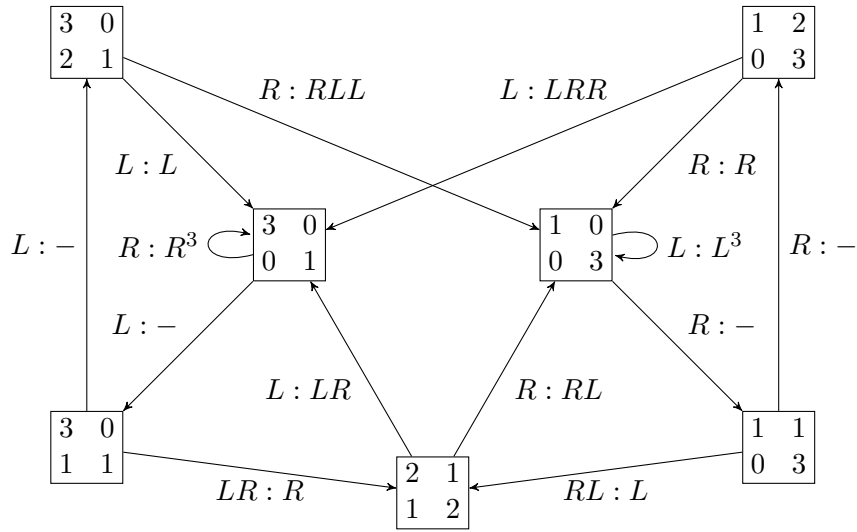
\begin{figure}
    \centering
    \begin{tikzpicture}[->,>=stealth',auto,node distance=1.5in]
      \tikzstyle{every state}=[rectangle,inner sep=0]
      \node[state] (0) 
      {$\begin{matrix}
          3 & 0 \\ 0 & 1
        \end{matrix}$};
      \path (0) ++(-60:1.5in) node [state] (1) 
      {$\begin{matrix}
          2 & 1 \\ 1 & 2
        \end{matrix}$};
      \node[state]         (2) [right of=0]
      {$\begin{matrix}
          1 & 0 \\ 0 & 3
        \end{matrix}$};
      \node[state] (0a) [below left of=0]
      {$\begin{matrix}
          3 & 0 \\ 1 & 1  
        \end{matrix}$};
      \node[state] (0b) [above left of=0]
      {$\begin{matrix}
          3 & 0 \\ 2 & 1  
        \end{matrix}$};
      \node[state] (2a) [below right of=2]
      {$\begin{matrix}
          1 & 1 \\ 0 & 3
        \end{matrix}$};
      \node[state] (2b) [above right of=2]
      {$\begin{matrix}
          1 & 2 \\ 0 & 3  
        \end{matrix}$};
      
      \path
      (0) edge [] node[above left] {$L:-$}  (0a)
      (0a) edge [] node {$L:-$}  (0b)
      (0b) edge [] node[below left] {$L:L$}  (0)
      (0) edge [loop left] node {$R:R^3$}  (0)
      (0a) edge []  node[below] {$LR:R$} (1)
      (0b) edge []  node[pos=.3,above right] {$R:RLL$} (2)
      (1) edge [] node {$L:LR$} (0)
      (1) edge [] node[below right] {$R:RL$} (2)
      (2) edge [loop right] node {$L:L^3$}  (2)
      (2) edge [] node {$R:-$}  (2a)
      (2a) edge [] node[right] {$R:-$}  (2b)
      (2b) edge [] node {$R:R$}  (2)
      (2a) edge []  node[below] {$RL:L$} (1)
      (2b) edge []  node[pos=.3,above left] {$L:LRR$} (0);
    \end{tikzpicture}
    \caption{The slow Raney transducer $\SRT_3$.}
    \label{fig:slow_Raney_3}
  \end{figure}
  
  Given $\xi \in [0,\infty]$, let $S_1S_2S_3\ldots$ be the $LR$-expansion of $\xi$ (or one of the two $LR$-expansions if $\xi$ is a positive rational). To compute $\gamma(\xi)$ for $\gamma \in \RB_n$, use the $S_i$ as directions for a walk along the graph,
  \begin{equation} \label{eq:inf_walk}
  \gamma \xrightarrow{S_1:W_1} 
  \gamma_1 \xrightarrow{S_2:W_2} 
  \gamma_2 \xrightarrow{S_3:W_3} \cdots
  \end{equation}
  to obtain a sequence of words $W_i$. We claim that the concatenation $W_1 W_2 W_3 \ldots$ is the $LR$-expansion of $\gamma(\xi)$. First, we show that the concatenation $W_1 W_2 W_3 \ldots$ is indeed an infinite word:
  \begin{lem}\label{lem:nonempty}
    Any walk
    \[
      \gamma_0 \xrightarrow{S_1:W_1} 
      \gamma_1 \xrightarrow{S_2:W_2} \dots
      \xrightarrow{S_n:W_n} \gamma_n
    \]
    of length $n$ on $\SRT_n$ has at least one nonempty output word $W_i \neq ()$.
  \end{lem}
  \begin{proof}
    Suppose that $W_1,\ldots, W_n$ are all empty. This means that
    \[
      1 \in \gamma_n (0,\infty) = \gamma_0 S_1 S_2 \cdots S_n (0,\infty).
    \]
    In particular, the number $\alpha = \gamma_0^{-1}(1)$ has unique $LR$-expansion starting with $S_1 S_2 \cdots S_n$. But, letting $\gamma = \bigl[\begin{smallmatrix}
        a & b \\ c & d
    \end{smallmatrix}\bigr]$, the rational number $\alpha = \frac{d - b}{a - c}$ has numerator and denominator at most $n$ and hence a finite $LR$-expansion of length at most $n-1$, which is a contradiction.
  \end{proof}
  
  Now, with $W_i$ as in \eqref{eq:inf_walk}, the number $\eta$ with $LR$-expansion $W_1 W_2 W_3 \ldots$ is characterized by
  \begin{align*}
    \{ \eta \} 
    &= \bigintsec_i W_1 W_2 \cdots W_i [0,\infty] \\
    &\subseteq \bigintsec_i W_1 W_2 \cdots W_i \gamma_i [0,\infty] \\
    &= \bigintsec_i \gamma S_1 S_2 \cdots S_i [0, \infty] \\
    &= \gamma \left(\bigintsec_i S_1 S_2 \cdots S_i [0, \infty] \right) \\
    &= \gamma (\{\xi\}) \\
    &= \{\gamma(\xi)\}.
  \end{align*}
  
  \subsection{The fast Raney transducer}
  As $n$ grows, the number of nodes grows rather quickly. A more compact alternative is a graph $\T_n$, which we call the \emph{(fast) Raney transducer,} whose nodes are the \emph{doubly balanced} matrices $\DB_n$ satisfying the inequalities
  \[
    a > b, \quad c < d, \quad a > c, \quad b < d
  \]
  (that is, neither row and neither column is dominant). Each edge has a label $V:W$ where both the input $V$ and the output $W$ are finite words in the alphabet $\{L,R\}$. As before, an edge
  \[
    \gamma \xrightarrow{V:W} \gamma'
  \]
  has the property that $\gamma \circ V = W \circ \gamma'$. The input words $V$ are no longer just one letter, but the input words emanating from any node form a \emph{base} for $LR$-sequences, that is, any infinite word in the letters $L$ and $R$ starts with exactly one of them, such as $\{L, RL, RRL, RRR\}$. Further details are found in Raney \cite{Raney_Automata}. The set $\DB_n$ is of manageable size; for $n = p$ prime, we have $\size{\DB_n} = p$. Examples are shown in Figures \ref{fig:Raney2}--\ref{fig:Raney13} (further examples can be found in \cite{Raney_Automata}). The reader is invited to prove the following construction of the fast transducer from the slow one (we will not need it):
  \begin{exercise}~
    \begin{enumerate}[(a)]
      \item A node $\gamma \in \RB_n$ of $\SRT_n$ has indegree $2$ or more if and only if $\gamma \in \DB_n$.
      \item Starting with $\SRT_n$, we repeatedly perform the following operation: Pick any node $\gamma$ with indegree $1$, merge it into its predecessor $\gamma'$, and update the labels of the edges out of $\gamma$ as shown:
      \[
        \gamma' \xrightarrow{V_1:W_1} \gamma \xrightarrow{V_2:W_2} \delta \quad \rightsquigarrow \quad
        \gamma' \xrightarrow{V_1V_2:W_1W_2} \delta
      \]
      When no nodes remain with indegree $1$, the resulting graph is $\T_n$.
    \end{enumerate}
  \end{exercise}

  \ignore{
    The arrangements of the nodes chosen emphasize the symmetry $\bigl[\begin{smallmatrix}
    a & b \\ c & d
  \end{smallmatrix}\bigr] \mapsto \bigl[\begin{smallmatrix}
    d & c \\ b & a
  \end{smallmatrix}\bigr] $, that is, conjugation by $\xi \mapsto 1/\xi$, which swaps the letters $L$ and $R$ in both input and output strings.

, and an overwhelming proportion of these are nodes with indegree $1$. If we merge each such node with its unique predecessor, we obtain a smaller graph, called a \emph{transducer} in Raney \cite{Raney_Automata}. 

    For instance, the Raney transducer of level $4$ computes the LFT
    \[
    \gamma(\xi) = \frac{2\xi + 1}{2} = \xi + \frac{1}{2},
    \]
    the subject of a recent MathOverflow question \cite{128676}.

    The Raney transducer has two symmetries. The first is conjugation by $\xi \mapsto 1/\xi$, which was noted earlier. The second, subtler symmetry sends each node $\bigl[\begin{smallmatrix}
    a & b \\ c & d
  \end{smallmatrix}\bigr]$ to $\bigl[\begin{smallmatrix}
    d & b \\ c & a
  \end{smallmatrix}\bigr] $, reverses arrows, and swaps input with output strings as well as reversing the order of the letters in each string.
  }

  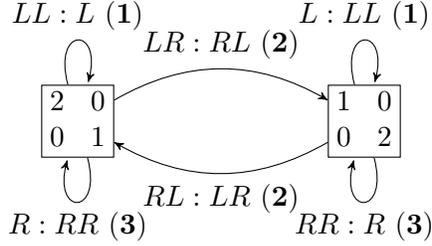
\begin{figure}
    \centering
    \begin{tikzpicture}[->,>=stealth',auto,node distance=1.5in]
      \tikzstyle{every state}=[rectangle,inner sep=0]
      \node[state] (0) 
      {$\begin{matrix}
          2 & 0 \\ 0 & 1
        \end{matrix}$};
      \node[state]         (1) [right of=0]
      {$\begin{matrix}
          1 & 0 \\ 0 & 2
        \end{matrix}$};
      
      \path (0) edge [loop above] node {$LL:L$ (\textbf{1})}  (0)
      edge [loop below] node {$R:RR$ (\textbf{3})}  (0)
      edge [bend left]  node {$LR:RL$ (\textbf{2})} (1)
      (1) edge [loop above] node {$L:LL$ (\textbf{1})}  (1)
      edge [loop below] node {$RR:R$ (\textbf{3})}  (1)
      edge [bend left]  node {$RL:LR$ (\textbf{2})} (0);
    \end{tikzpicture}
    \caption{The Raney transducer $\T_2$. The parenthesized boldface numbers are the corresponding \emph{Romik digits}, used in \cite{RomikDynamics,KimSim}. The lowest point $\min \L_2 = 2\sqrt{2}$ arises by following the $2$-cycle of edges marked (\textbf{2}).}
    \label{fig:Raney2}
  \end{figure}
  
  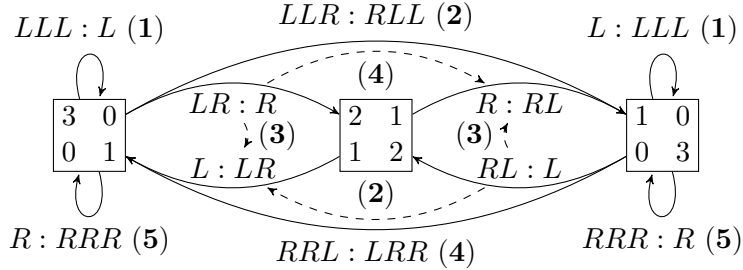
\begin{figure}
    \centering
    \begin{tikzpicture}[->,>=stealth',auto,node distance=1.5in]
      \tikzstyle{every state}=[rectangle,inner sep=0]
      \node[state] (0) 
      {$\begin{matrix}
          3 & 0 \\ 0 & 1
        \end{matrix}$};
      \node[state]         (1) [right of=0]
      {$\begin{matrix}
          2 & 1 \\ 1 & 2
        \end{matrix}$};
      \node[state]         (2) [right of=1]
      {$\begin{matrix}
          1 & 0 \\ 0 & 3
        \end{matrix}$};
      
      \path (0) edge [loop above] node {$LLL:L$ (\textbf{1})}  (0)
      edge [loop below] node {$R:RRR$ (\textbf{5})}  (0)
      edge [bend left]  node[below] (lrxr) {$LR:R$} (1)
      edge [bend left]  node {$LLR:RLL$ (\textbf{2})} (2)
      (1) edge [bend left] node[above] (lxlr) {$L:LR$} (0)
      (1) edge [bend left] node[below] (rxrl) {$R:RL$} (2)
      (2) edge [loop above] node {$L:LLL$ (\textbf{1})}  (2)
      edge [loop below] node {$RRR:R$ (\textbf{5})}  (2)
      edge [bend left]  node[above] (rlxl) {$RL:L$} (1)
      edge [bend left]  node {$RRL:LRR$ (\textbf{4})} (0);

      \path (lrxr) edge [dashed,bend left] node[below] {(\textbf{4})} (rxrl);
      \path (rlxl) edge [dashed,bend left] node[above] {(\textbf{2})} (lxlr);
      \path (lrxr) edge [dashed,bend left] node[right] {(\textbf{3})} (lxlr);
      \path (rlxl) edge [dashed,bend left] node[left]  {(\textbf{3})} (rxrl);
      
    \end{tikzpicture}
    \caption{The Raney transducer $\T_3$. The lowest point $\min \L_3 = 2\sqrt{3}$ arises by following either of the two small $2$-cycles involving the middle node.}
    \label{fig:Raney3}
  \end{figure}
  
  \begin{figure}
    \centering
    \begin{tikzpicture}[,>=stealth',node distance=0.75in]
      \tikzstyle{every state}=[inner sep=0]
      \node[state] (2)               {$2$};
      \node[state, right of=2] (5)   {$5$};
      \node[state, right of=5]  (3)  {$3$};
      \node[state, below left=1.125in and 0.375in of 2, anchor=south west]  (0)  {$0$};
      \node[state, right of=0]  (1) {$1$};
      \node[state, right of=1] (11)  {$11$};
      \node[state, right of=11] (12)   {$12$};
      \node[state, below left=1.125in and 0.375in of 1, anchor=south west]  (10) {$10$};
      \node[state, right of=10] (7)   {$7$};
      \node[state, right of=7]  (9)  {$9$};
      \node[state, below left of=0]  (6)  {$6$};
      \node[state, below right of=12] (8)   {$8$};
      \node[state, above left of=0]  (4)  {$4$};
      
      \path
      (0) edge[double,loop left] (0)
      (0) edge[double] (1)
      (0) edge[double] (2)
      (0) edge[double] (3)
      (0) edge[double] (4)
      (0) edge[double] (5)
      (0) edge[double] (6)
      (0) edge[double] (7)
      (0) edge[double] (8)
      (0) edge[double] (9)
      (0) edge[double] (10)
      (0) edge[double,bend left] (11)
      (0) edge[double,bend right] (12)
      (1) edge[->] (2)
      (1) edge[->] (3)
      (1) edge[->] (4)
      (1) edge[->] (5)
      (1) edge[double] (11)
      (1) edge[double,bend left] (12)
      (2) edge[->,bend left] (3)
      (2) edge[double] (5)
      (2) edge[->] (11)
      (3) edge[->] (5)
      (3) edge[->] (11)
      (4) edge[->] (11)
      (5) edge[->] (11)
      (7) edge[->] (1)
      (7) edge[double] (10)
      (8) edge[->] (1)
      (9) edge[->] (1)
      (9) edge[->] (7)
      (10) edge[->] (1)
      (10) edge[->,bend right] (9)
      (11) edge[->] (7)
      (11) edge[->] (8)
      (11) edge[->] (9)
      (11) edge[->] (10)
      (12) edge[double] (2)
      (12) edge[double] (3)
      (12) edge[double,bend right] (4)
      (12) edge[double] (5)
      (12) edge[double,bend left] (6)
      (12) edge[double] (7)
      (12) edge[double] (8)
      (12) edge[double] (9)
      (12) edge[double] (10)
      (12) edge[double] (11)
      (12) edge[double,loop right] (12);
    \end{tikzpicture}
    
    \begin{tabular}{|c@{}r@{${}:{}$}l@{}c|}
      From & In & Out & To \\ \hline
      $0$ & $L^{13}$ & $L$ & $0$ \\
      $0$ & $R$ & $R^{13}$ & $0$ \\
      $0$ & $LR$ & $R^6$ & $1$ \\
      $0$ & $L^7R$ & $RLR$ & $2$ \\
      $0$ & $L^9R$ & $RLL$ & $3$ \\
      $0$ & $L^8R$ & $RL$ & $4$ \\
      $0$ & $L^{10}R$ & $RLLL$ & $5$ \\
      $0$ & $L^6R$ & $R$ & $6$ \\
      $0$ & $LLR$ & $R^4$ & $7$ \\
      $0$ & $L^4R$ & $RR$ & $8$ \\
      $0$ & $LLLR$ & $RRR$ & $9$ \\
      $0$ & $L^5R$ & $RRL$ & $10$ \\
      $0$ & $L^{11}R$ & $RL^5$ & $11$ \\
      $0$ & $L^{12}R$ & $RL^{12}$ & $12$ \\ \hline
      $1$ & $L^6$ & $LR$ & $0$ \\
      $1$ & $R$ & $RR$ & $2$ \\
      $1$ & $LLR$ & $RL$ & $3$ \\
      $1$ & $LR$ & $R$ & $4$ \\
      $1$ & $LLLR$ & $RLL$ & $5$ \\
      $1$ & $L^4R$ & $RL^4$ & $11$ \\
      $1$ & $L^5R$ & $RL^{11}$ & $12$ \\ \hline
      $2$ & $L^4$ & $LRR$ & $0$ \\
      $2$ & $R$ & $R$ & $3$ \\
      $2$ & $LR$ & $RL$ & $5$ \\
      $2$ & $LLR$ & $RLLL$ & $11$ \\
      $2$ & $LLLR$ & $RL^{10}$ & $12$ \\
    \end{tabular}
    \begin{tabular}{|c@{}r@{${}:{}$}l@{}c|}
      From & In & Out & To \\ \hline
      $3$ & $LLL$ & $LRRR$ & $0$ \\
      $3$ & $R$ & $R$ & $5$ \\
      $3$ & $LR$ & $RLL$ & $11$ \\
      $3$ & $LLR$ & $RL^9$ & $12$ \\ \hline
      $4$ & $LL$ & $LR^4$ & $0$ \\
      $4$ & $R$ & $RL$ & $11$ \\
      $4$ & $LR$ & $RL^8$ & $12$ \\ \hline
      $5$ & $RLL$ & $LR^5$ & $0$ \\
      $5$ & $L$ & $L$ & $2$ \\
      $5$ & $RR$ & $R$ & $11$ \\
      $5$ & $RLR$ & $RL^7$ & $12$ \\ \hline
      $6$ & $L$ & $LR^6$ & $0$ \\
      $6$ & $R$ & $RL^6$ & $12$ \\ \hline
      $7$ & $LRL$ & $LR^7$ & $0$ \\
      $7$ & $LL$ & $L$ & $1$ \\
      $7$ & $R$ & $R$ & $10$ \\
      $7$ & $LRR$ & $RL^5$ & $12$ \\ \hline
      $8$ & $RL$ & $LR^8$ & $0$ \\
      $8$ & $L$ & $LR$ & $1$ \\
      $8$ & $RR$ & $RL^4$ & $12$ \\ \hline
      $9$ & $RRL$ & $LR^{9}$ & $0$ \\
      $9$ & $RL$ & $LRR$ & $1$ \\
      $9$ & $L$ & $L$ & $7$ \\
      $9$ & $RRR$ & $RLLL$ & $12$
    \end{tabular}
    \begin{tabular}{|c@{}r@{${}:{}$}l@{}c|}
      From & In & Out & To \\ \hline
      $10$ & $RRRL$ & $LR^{10}$ & $0$ \\
      $10$ & $RRL$ & $LRRR$ & $1$ \\
      $10$ & $RL$ & $LR$ & $7$ \\
      $10$ & $L$ & $L$ & $9$ \\
      $10$ & $R^4$ & $RLL$ & $12$ \\ \hline
      $11$ & $R^5L$ & $LR^{11}$ & $0$ \\
      $11$ & $R^4L$ & $LR^4$ & $1$ \\
      $11$ & $RRRL$ & $LRR$ & $7$ \\
      $11$ & $RL$ & $L$ & $8$ \\
      $11$ & $RRL$ & $LR$ & $9$ \\
      $11$ & $L$ & $LL$ & $10$ \\
      $11$ & $R^6$ & $RL$ & $12$ \\ \hline
      $12$ & $R^{12}L$ & $LR^{12}$ & $0$ \\
      $12$ & $R^{11}L$ & $LR^5$ & $1$ \\
      $12$ & $R^5L$ & $LLR$ & $2$ \\
      $12$ & $RRRL$ & $LLL$ & $3$ \\
      $12$ & $R^4L$ & $LL$ & $4$ \\
      $12$ & $RRL$ & $L^4$ & $5$ \\
      $12$ & $R^6L$ & $L$ & $6$ \\
      $12$ & $R^{10}L$ & $LRRR$ & $7$ \\
      $12$ & $R^8L$ & $LR$ & $8$ \\
      $12$ & $R^{9}L$ & $LRR$ & $9$ \\
      $12$ & $R^7L$ & $LRL$ & $10$ \\
      $12$ & $RL$ & $L^6$ & $11$ \\
      $12$ & $L$ & $L^{13}$ & $12$ \\
      $12$ & $R^{13}$ & $R$ & $12$
    \end{tabular}
    
    \vspace*{0.5em}
    \begin{tabular}{@{}r@{ }|@{}c@{}c@{}c@{}c@{}c@{}c@{}c@{}c@{}c@{}c@{}c@{}c@{}c@{}}
      Label & $0$ & $1$ & $2$ & $3$ & $4$ & $5$ & $6$ & $7$ & $8$ & $9$ & $10$ & $11$ & $12$ \\ \cline{1-1}
      Matrix & $\begin{bmatrix}
        13 & 0 \\
        0 & 1
      \end{bmatrix}$ &
      $\begin{bmatrix}
        7 & 1 \\
        1 & 2
      \end{bmatrix}$ &
      $\begin{bmatrix}
        5 & 2 \\
        1 & 3
      \end{bmatrix}$ &
      $\begin{bmatrix}
        4 & 3 \\
        1 & 4
      \end{bmatrix}$ &
      $\begin{bmatrix}
        5 & 4 \\
        3 & 5
      \end{bmatrix}$ &
      $\begin{bmatrix}
        3 & 2 \\
        1 & 5
      \end{bmatrix}$ &
      $\begin{bmatrix}
        7 & 6 \\
        6 & 7
      \end{bmatrix}$ &
      $\begin{bmatrix}
        5 & 1 \\
        2 & 3
      \end{bmatrix}$ &
      $\begin{bmatrix}
        5 & 3 \\
        4 & 5
      \end{bmatrix}$ &
      $\begin{bmatrix}
        4 & 1 \\
        3 & 4
      \end{bmatrix}$ &
      $\begin{bmatrix}
        3 & 1 \\
        2 & 5
      \end{bmatrix}$ &
      $\begin{bmatrix}
        2 & 1 \\
        1 & 7
      \end{bmatrix}$ &
      $\begin{bmatrix}
        1 & 0 \\
        0 & 13
      \end{bmatrix}$
    \end{tabular}
    \caption{The Raney transducer $\T_{13}$.}
    \label{fig:Raney13}
  \end{figure}

  \subsection{Connection to Romik expansions}
  In \cite{KimSim} and \cite{ChaEis}, respectively, \emph{Romik expansions} of real numbers are used to understand the spectra here denoted $\L_2$ and $\L_3$. This notion is closely related to Raney transducers of small level. For $\T_2$ (Figure \ref{fig:Raney2}), the three outgoing edges of each node correspond to the Romik digits \textbf{1}, \textbf{2}, and \textbf{3}. For $\T_3$ (Figure \ref{fig:Raney3}), the generalized Romik digits \textbf{1}--\textbf{5} of \cite{ChaEis} correspond either to edges or to two-edge walks (the latter indicated by the dashed arrows) starting and ending at one of the two hub nodes $\bigl[\begin{smallmatrix}
    3 & 0 \\ 0 & 1
  \end{smallmatrix}\bigr]$, $\bigl[\begin{smallmatrix}
    1 & 0 \\ 0 & 3
  \end{smallmatrix}\bigr]$. Any walk on $\T_3$ can be translated into a sequence of Romik digits thanks to the fact that the walk must return at least every two steps to one of these two hub nodes.

  For larger $n$, the picture is quite different, as shown in Figure \ref{fig:Raney13}. Here there are many infinite walks not meeting the two hub nodes labeled $0$ and $12$. In particular, the lowest point $\min \L_{13} = \sqrt{221}/5$ corresponds to a (non-simple!) cycle
  \[
    2 \xrightarrow{R:R} 3 \xrightarrow{R:R} 5 \xrightarrow{L:L} 2 \xrightarrow{LR:RL} 5 \xrightarrow{L:L} 2 
  \]
  (or to its mirror image $10 \to 9 \to 7 \to 10 \to 7 \to 10$) that does not use the hub nodes $0$, $12$ at all. Consequently, we do not expect a system of Romik digits to shed light on $\L_n$ for $n > 3$.

  \begin{rem}
    The dichotomy between the $n \leq 3$ and $n \geq 4$ cases can also be explained geometrically in terms of the images of the geodesic $0$---$i\infty$ under $\Gamma_0(n)$. For $n = 1,2,3$ the resulting geodesics tessellate the hyperbolic plane by congruent regular $3$-, $4$-, and $6$-gons, respectively (see Figures \ref{figure:hyp_tiling1}, \ref{figure:hyp_tiling2}, \ref{figure:hyp_tiling3}), suggesting encodings with $2$, $3$, and $5$ symbols respectively (because a line entering a tile at one side can exit on any of the other sides). For $n \geq 4$, the resulting geodesics do not form regions with finitely many sides (see Figure \ref{figure:hyp_tiling5}).
  \end{rem}

  \begin{figure}[ht]
    \centering
    \begin{tikzpicture}[xscale=7,yscale=7]
      \node[below] at (0,0) (0) {$0$};
      \node[below] at (1,0) (0) {$1$};
      \mirrorArc{0/1}{1/2}
      \mirrorArc{0/1}{1/4}
      \mirrorArc{0/1}{1/6}
      \mirrorArc{0/1}{1/8}
      \mirrorArc{0/1}{1/10}
      \mirrorArc{0/1}{1/12}
      \mirrorArc{1/12}{1/11}
      \mirrorArc{1/11}{1/10}
      \mirrorArc{1/10}{1/9}
      \mirrorArc{1/9}{1/8}
      \mirrorArc{1/8}{1/7}
      \mirrorArc{1/7}{1/6}
      \mirrorArc{1/6}{1/5}
      \mirrorArc{1/5}{1/4}
      \mirrorArc{1/4}{1/3}
      \mirrorArc{1/3}{1/2}
      \mirrorArc{1/3}{3/8}
      \mirrorArc{3/8}{2/5}
      \mirrorArc{2/5}{1/2}
      \draw
      (0,0) -- (0,0.4)
      (0,0) -- (1,0) -- (1,0.4);
    \end{tikzpicture}
    \caption{Tiling by hyperbolic squares corresponding to $\Gamma^0(2)$}
    \label{figure:hyp_tiling2}
  \end{figure}
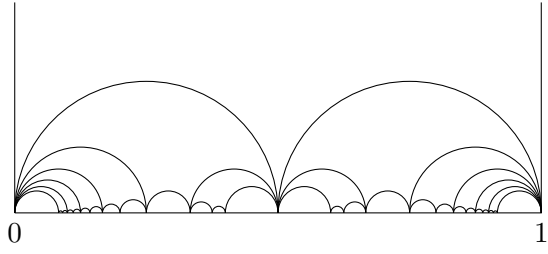

  \begin{figure}[ht]
    \centering
    \begin{tikzpicture}[xscale=7,yscale=7]
      \node[below] at (0,0) (0) {$0$};
      \node[below] at (1,0) (0) {$1$};
      \mirrorArc{0/1}{1/3}
      \mirrorArc{0/1}{1/6}
      \mirrorArc{0/1}{1/9}
      \mirrorArc{1/9}{1/8}
      \mirrorArc{1/8}{2/15}
      \mirrorArc{2/15}{1/7}
      \mirrorArc{1/7}{1/6}
      \mirrorArc{1/6}{1/5}
      \mirrorArc{1/5}{2/9}
      \mirrorArc{2/9}{1/4}
      \mirrorArc{1/4}{1/3}
      \mirrorArc{1/3}{1/2}
      \mirrorArc{1/3}{2/5}
      \mirrorArc{2/5}{5/12}
      \mirrorArc{5/12}{3/7}
      \mirrorArc{3/7}{4/9}
      \mirrorArc{4/9}{1/2}
      \draw
      (0,0) -- (0,0.4)
      (0,0) -- (1,0) -- (1,0.4);
    \end{tikzpicture}
    \caption{Tiling by hyperbolic regular hexagons corresponding to $\Gamma^0(3)$}
    \label{figure:hyp_tiling3}
  \end{figure}
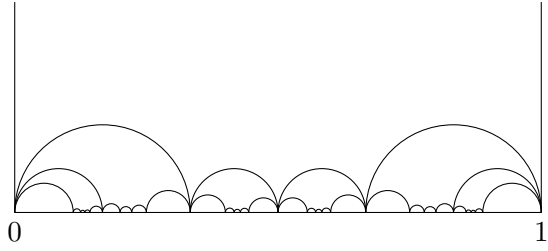

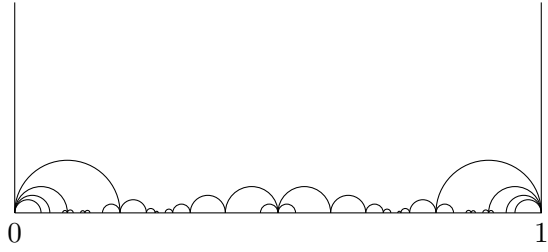
\begin{figure}[ht]
  \centering
  \begin{tikzpicture}[xscale=7,yscale=7]
    \node[below] at (0,0) (0) {$0$};
    \node[below] at (1,0) (0) {$1$};
    \mirrorArc{0/1}{1/5}
    \mirrorArc{0/1}{1/10}
    \mirrorArc{0/1}{1/15}
    \mirrorArc{0/1}{1/20}
    \mirrorArc{1/11}{1/10}
    \mirrorArc{1/10}{1/9}
    \mirrorArc{1/6}{1/5}
    \mirrorArc{1/5}{1/4}
    \mirrorArc{1/3}{2/5}
    \mirrorArc{2/5}{1/2}
    \mirrorArc{2/7}{3/10}
    \mirrorArc{3/10}{1/3}
    \mirrorArc{2/15}{1/7}
    \mirrorArc{1/8}{2/15}
    \mirrorArc{4/15}{3/11}
    \mirrorArc{1/4}{4/15}
    \mirrorArc{7/15}{1/2}
    \draw
    (0,0) -- (0,0.4)
    (0,0) -- (1,0) -- (1,0.4);
  \end{tikzpicture}
  \caption{Tiling  corresponding to $\Gamma^0(5)$. The tiles have infinitely many sides.}
  \label{figure:hyp_tiling5}
\end{figure}
  
  \section{Closedness of the \texorpdfstring{$n$}{n}-Lagrange spectrum}
  \label{sec:closed}

  \begin{prop} \label{prop:Ln_closed}
    For each $n \geq 1$, $\L_n$ is closed.
  \end{prop}
  The proof given below closely follows the proof of closedness of the classical Lagrange spectrum given by Cusick \cite{CusickConnection}, and in particular shows the following stronger result.
  \begin{prop}
    For each $n \geq 1$, $\L_n$ is the closure of the set $\P_n$ of $n$-approximabilities of quadratic irrationals.
  \end{prop}
  \begin{proof}
    Let $G$ be a finite graph that computes, for given $\xi$, the $LR$-expansion of $d\xi$ for all divisors $d\mid n$ simultaneously. Such a graph can be constructed as follows. Take for the nodes $V(G)$ the cartesian product $\prod_{d|n} V(\SRT_d)$ of the nodes of the slow Raney transducers corresponding to the divisors of $n$. Give each node $(\gamma_d)_d$ two outgoing edges
    \[
    (\gamma_d)_d \xrightarrow{S : (W_d)_d} (\gamma'_d)_d
    \]
    whose targets and labels are derived from those of the edges
    \[
    \gamma_d \xrightarrow{S : W_d} \gamma'_d
    \]
    emanating from the respective node $\gamma_d$ in each graph $\S_d$.
    
    Now any infinite $LR$-sequence can be encoded as a walk on $G$ starting from the start node $\bigl(\bigl[\begin{smallmatrix}
      d & 0 \\ 0 & 1
    \end{smallmatrix}\bigr]\bigr)_{d\mid n}$. As we walk, we assemble the words $W_d$ for each divisor $d\mid n$ to get the $LR$-expansion of $d\xi$. We can then compute $\lambda_n(\xi)$ by taking the limsup of the approximation qualities of each cut of each $d\xi$ via Proposition \ref{prop:quality}.
    
    We now proceed to prove that
    \begin{equation} \label{eq:closure}
      \L_n = \operatorname{Closure}(\P_n).
    \end{equation}
    Observe that $\P_n$ is the subset of approximabilities derived from (eventually) \emph{periodic} walks on $G$. We prove the $(\subseteq)$ and $(\supseteq)$ directions separately.
    
    \paragraph{$(\subseteq)$} Let $\alpha = \lambda_n(\xi)$. Then $\alpha = \lambda(d_0\xi)$ for some $d_0 \mid n$ by Proposition \ref{prop:lambda_n->lambda}. View $\xi$ as an infinite walk
    \[
      \W \colon \quad v_0 \xrightarrow[e_0]{S_1:(W_{1,d})_d} 
      v_1 \xrightarrow[e_1]{S_2:(W_{2,d})_d} 
      v_2 \xrightarrow[e_2]{S_3:(W_{3,d})_d} \cdots.
    \]
    Inasmuch as an $LR$-sequence is another form of a continued fraction, we define a \emph{cut} of an $LR$-sequence to be a choice of a (maximal) consecutive block of $L$'s or $R$'s, and the \emph{quality} of the cut to be the quality of the corresponding cut of a continued fraction, as explained in Section \ref{sec:classical}.

    There is a sequence of cuts $C_1,C_2,\ldots$ of the $LR$-expansion of $d_0\xi$ whose qualities converge to $\alpha$. A cut $C$ arises from a block of consecutive $L$'s or $R$'s on the output words $W_{i,d}$ in $G$ giving the $LR$-expansion of $d\xi$ for some $d \mid n$. Define the \emph{cut edge} to be the edge $e_{C}$ of $G$ on which the first $L$ or $R$ of this block is written. Since $G$ is a finite graph, there is for each $k \geq 1$ and $\epsilon > 0$ a pair of cuts $C^{(1)}, C^{(2)}$ such that:
    \begin{itemize}
      \item $C^{(1)}$ has quality $\geq \alpha - \epsilon$.
      \item All cuts of all $d\xi$, $d\mid n$ whose cut edge are at or after the cut edge $e_i$ of $C^{(1)}$ along $\W$, have quality $\leq \alpha + \epsilon$.
      \item $C^{(2)}$ has the same cut edge $e_j = e_i$ as $C^{(1)}$ and corresponds to the same block of $L$'s or $R$'s on the label $W_{i,d}$ of this edge, but occurring later on in $\W$: $j > i$.
      \item Moreover, the $k$ steps before and after the cut edge agree for $C^{(1)}$ and $C^{(2)}$: $e_{i + h} = e_{j + h}$, $-k \leq h \leq k$.
    \end{itemize}
    Also, since $\alpha = \lambda_n(\xi)$ is finite, the terms of the continued fractions of all $d\xi$ are eventually bounded by $\floor{\alpha}$, and we can assume that all the terms from $e_{i-k}$ onward obey this bound.

    Let $\xi'$ be the irrational number corresponding to the following walk $\W'$ on $G$: Begin as for $\W$ until reaching the cut edge $e_j$, and then repeat the portion of $\W$ following $e_i$ up to $e_j$ infinitely. Note that $\W'$ is a periodic walk, so $\xi'$ is a quadratic irrational. By construction, for any cut $C'$ of $\W'$, there is a cut $C$ of $\W$ such that $C'$ and $C$ agree for at least $k$ edges before and after the cut edge. By Lemma \ref{lem:nonempty}, the corresponding $LR$-sequences agree before and after the first letter of the cut for at least $k/n$ letters, yielding at least $k/(n\floor{\alpha}) - 1$ common terms of the continued fraction on each side of the cut term. By Lemma \ref{lem:sensitive}, we have
    \[
      \lambda(C') \leq \lambda(C) + 2^{3-\frac{k}{n\floor{\alpha}}} \leq \alpha + \epsilon + 2^{3-\frac{k}{n\floor{\alpha}}}.
    \]
    Also, the recurring appearance of $e_i$ yields an infinite sequences of cuts $C'$ for which 
    \[
      \lambda(C') \geq \lambda(C^{(1)}) - 2^{3-\frac{k}{n\floor{\alpha}}} \geq \alpha - \epsilon - 2^{3-\frac{k}{n\floor{\alpha}}}.
    \]
    Consequently,
    \begin{equation} \label{eq:close}
      \size{\lambda_n(\xi') - \alpha} \leq \epsilon + 2^{3-\frac{k}{n\floor{\alpha}}}.
    \end{equation}
    Taking $k \to \infty$ and $\epsilon \to 0$, the right-hand side of \eqref{eq:close} goes to $0$, so $\alpha \in \operatorname{Closure}(\P_n)$ as desired.
    
    \paragraph{$(\supseteq)$} Let $\xi_1, \xi_2, \ldots$ be a sequence of quadratic irrationals whose $n$-approximabilities $\alpha_i$ converge to $\alpha \in \operatorname{Closure}(\P_n)$. For each $i$, $\alpha_i$ is the quality of a cut $C_i$ in the periodic part of $d_i\xi_i$, according to \eqref{eq:periodic}. Also, $\xi_i$ corresponds to an eventually periodic walk $\W_i$ on $G$, and $C_i$ has a particular cut edge $e_i$ in the periodic part of $\W_i$. Passing to a subsequence, we may assume that:
    \begin{itemize}
      \item All $d_i = d$ are equal.
      \item All $e_i = e$ are the same edge, and the cuts $C_i$ are at the same block of $L$'s or $R$'s starting on this edge.
      \item The walks $\W_k$ and $\W_{k+1}$ agree for $k$ steps before and after $e$.
    \end{itemize}
    Now construct a walk $\W$ as follows. Follow $\W_1$ until the cut edge $e$, then walk one period of $\W_2$ starting and ending at $e$, then one period of $\W_3$ starting and ending at $e$, and so on. This $\W$ corresponds to an irrational number $\xi$. Similar to the preceding part, we check that $\lambda_n(\xi) = \alpha$, establishing that $\alpha \in \L_n$.
  \end{proof}
  
  
  \section{An algorithm to compute \texorpdfstring{$\min \L_p$}{min ℒₚ}}
  \label{sec:alg}
  
  Raney transducers are ideally suited for computing elements of the $n$-Lagrange and $n$-Markoff spectra, especially if $n = p$ is prime, to which case we now specialize. 
  

  Recall that an $LR$-sequence can be viewed as a continued fraction, where each term of the continued fraction corresponds to a run of consecutive $L$'s or $R$'s in the sequence. Given a finite $LR$-sequence
  \[
  W = L^{a_0} R^{a_1} L^{a_2} \cdots (L\text{ or }R)^{a_k} \textor R^{a_0} L^{a_1} R^{a_2} \cdots (L\text{ or }R)^{a_k},
  \]
  let $\underline\lambda(W)$ be the minimum possible approximability of a completion of $W$, as evidenced by cuts in $W$: namely
  \[
  \underline\lambda(W) = \max_{1 \leq i \leq k} \bigl(a_i + [0, a_{i-1}, a_{i-2}, \ldots, a_{i \mathbin{\%} 2}] + [0, a_{i+1}, a_{i+2}, \ldots, a_{k - ((k - i) \mathbin{\%} 2)}] \bigr),
  \]
  where $i \mathbin{\%} 2$ denotes the least nonnegative residue of $i$ modulo $2$ (either $0$ or $1$).
  
  Let $\beta > 0$. We say that an $LR$-sequence $S$ is \emph{$\beta$-good} if $\underline\lambda(S) \leq \beta$, and \emph{$\beta$-bad} otherwise. We say that a path $P$ of a (slow or fast) Raney transducer is \emph{$\beta$-good} if the input and output $LR$-sequences $I_P, O_P$ formed by walking this path are both $\beta$-good; and \emph{$\beta$-bad} otherwise.
  
  We can now state our algorithm for computing $\min \L_n$, currently implemented only for primes $n = p$.
  
  \begin{alg}~\label{alg:main}
    \begin{itemize}
      \item[] \textbf{Input:}
      \begin{itemize}
        \item A prime $p$
      \end{itemize}
      \item[] \textbf{Output, if the algorithm terminates:}
      \begin{itemize}
        \item The minimal value $\alpha = \min \L_p$
        \item A bound $\beta$ such that $\L_p \intsec [0, \beta) = \{\alpha\}$
        \item A listing of representatives of all $\Gamma^0(p)$-classes of irrationals $\xi$ achieving $\lambda_p(\xi) = \alpha$
      \end{itemize}
      \item[] \textbf{Running variables:}
      \begin{itemize}
        \item $k$, the lengths of paths to be considered (initially $1$)
        \item $\alpha$, the lowest known value in $\L_p$ (initially $\infty$)
        \item $\beta$, the lowest $\underline\lambda$-value found for $\alpha$-bad paths (initially $\infty$)
        \item A list of all good paths of length $k - 1$ (initially, all paths of length $0$ are good)
      \end{itemize}
      \item[] \textbf{Iteration step, for $k = 1,2,\ldots\colon$}
      \begin{enumerate}
        \item List all paths of length $k$ on the Raney transducer $\T_p$ whose subpaths of length $k-1$ are both $\alpha$-good.
        \item For each such path $P$, check whether $P$ is $\alpha$-good, that is, its input and output words $V,W$ satisfy $\underline\lambda(P) = \max\{\underline\lambda(V), \underline\lambda(W) \leq \alpha$.
        \item If $P$ is $\alpha$-good, add $P$ to the list of good paths. \item Additionally, if $P$ is an $\alpha$-good cycle, update $\alpha \leftarrow \min\{\alpha, \lambda(P^\infty)\}$, where $\lambda(P^\infty) = \max\{\lambda(V^\infty), \lambda(W^\infty)\}$ denotes the approximability of $P$ extended periodically as in \eqref{eq:periodic}.
        \item If $P$ is $\alpha$-bad, let $\beta \leftarrow \min\{\beta, \underline\lambda(P)\}$.
      \end{enumerate}
      \item[] \textbf{Stopping condition:}
      We stop if, for some positive integers $m$ and $t$, the good paths have the following property: every path of length $m$ occurring as the middle segment of a good path of length $m + 2t$ has a unique forward extension to a path of length $m+1$ occurring as the middle segment of a good path of length $m + 2t + 1$. If this happens, we deduce that any infinite good path repeats (after at most $t$ transient initial moves) around one of finitely many cycles $P_1,\ldots, P_r$. We compute their approximabilities $\alpha_i = \lambda(P_i^\infty)$ and arrange them so that $\alpha_1\leq \cdots \leq \alpha_r$. Let $\alpha_2'$ be the smallest $\alpha_i$ that exceeds the minimum $\alpha_1$ ($\alpha_2' = \infty$ if no such $\alpha_i$ is found).
      \item[] \textbf{Return values:} Once the stopping condition is achieved, we return
        \begin{itemize}
          \item $\alpha = \alpha_1$: the best $\lambda_p$-value of a cycle found
          \item $\beta \leftarrow \min\{\beta, \alpha_2'\}$
          \item an irrational $\xi$ corresponding to each optimal cycle, which may be retrieved by picking a path from the start node $\bigl[\begin{smallmatrix}
            p & 0 \\ 0 & 1
          \end{smallmatrix}\bigr]$ to the cycle and converting the resulting eventually periodic $LR$-expansion to an irrational number $\xi$. (This is possible because the Raney transducer is strongly connected, as Raney proves \cite[8.3]{Raney_Automata}.)
        \end{itemize}
    \end{itemize}
  \end{alg}
  \begin{rem}
    The choice of $t$ and hence $m$ in the stopping condition is unimportant; it is easy to see that if the stopping condition holds for some $m$ and $t$, it holds for all sufficiently large $m$ and $t$. In our implementation we take $t \approx \sqrt{k}$, $m = k - 2t - 1$ which seems to be good enough.
  \end{rem}
  \begin{thm} \label{thm:alg_OK}
    If this algorithm terminates, it shows that 
    \begin{enumerate}[(a)]
      \item $\alpha = \min \L_p$ is the square root of a rational number, and
      \item $\L_p \intsec [-\infty, \beta) = \{\alpha\}$; in particular, $\alpha$ is an isolated point of $\L_n$ and $\M_n$.
    \end{enumerate}
  \end{thm}
  \begin{proof}
    If the stopping condition is satisfied for some $t$ and $m$, then already a cycle $P$ has been found of a finite approximability $\alpha = \lambda(P^\infty)$. Its approximability $\lambda_n(P^\infty)$ is always the square root of a rational number for the reasons pointed out in Section \ref{sec:quad_approx}. Let $\beta'$ be the value of $\beta$ just before the final step, where we return $\beta = \min\{\beta', \alpha_2'\}$. Let $\gamma = \lambda_n(\xi) \in \L_n$ be a value. We may assume that $\xi$ is positive and thus corresponds to an infinite walk $W$ on $\T_p$. There are two cases:
    \begin{enumerate}
      \item There are infinitely many paths of length $m$ in the walk $W$ that are \emph{not} the middle segment of a good path of length $m + 2t$. Each such path corresponds to an approximation to $\xi$ whose $n$-quality is at least $\beta$, implying $\gamma \geq \beta'$.
      \item After a certain point, every path of length $m$ in the walk $W$ is the middle segment of a good path of length $m + 2t$. By the stopping condition, $W$ is eventually periodic, with period being one of the cycles enumerated at the last step. So $\gamma$ is one of $\alpha_1,\ldots, \alpha_t$.
    \end{enumerate}
    Hence $\L_n \intsec [0,\beta') = \{\alpha_1,\ldots, \alpha_t\} \intsec [0,\beta')$. In particular, the lowest point $\alpha_1$ of $\L_n$ is isolated, and the next lowest point is at least $\min\{\alpha_2', \beta'\} = \beta$, as desired.

    For each good cycle $\alpha$, the corresponding irrational $\xi$ is unique up to changing finitely many initial steps, which corresponds to a transformation in $\Gamma^0(p)$ for the following reason. If we have two eventually identical paths
    \[
      \begin{tikzpicture}[->,>=latex]
        \node (g0) at (0,0) {$\gamma_0$};
        \node (g) at (2,0) {$\gamma$};
        \node (dots) at (4,0) {$\cdots$};
        \path (g0) edge [bend left] node[above] {$V_1 : W_1$} (g);
        \path (g0) edge [bend right] node[below] {$V_2 : W_2$} (g);
        \path (g) edge node[above] {$V:W$} (dots);
      \end{tikzpicture}
    \]
    then the corresponding real numbers $\xi_1, \xi_2$ with LR-expansions $V_1V, V_2V$ differ by the LFT
    \[
      \xi_2 = V_2 V_1^{-1} \xi_1,
    \]
    where, by the defining relation of the Raney transducer,
    \begin{align*}
      V_2 V_1^{-1} &= (\gamma_0^{-1} W_2 \gamma) (\gamma_0^{-1} W_1 \gamma)^{-1} \\
      &= \gamma_0^{-1} W_2 W_1^{-1} \gamma_0 \\
      &\in \SL_2 \ZZ \intsec \gamma_0^{-1} \cdot \SL_2 \ZZ \cdot \gamma_0 \\
      &= \Gamma^0(p). \qedhere
    \end{align*}
  \end{proof}
  \begin{rem}
    Together with the computations in the attached code, this shows Theorems \ref{thm:main} and \ref{thm:67}. It would be attractive to prove that, conversely, if the lowest point is isolated in one or both spectra then the algorithm terminates; but this does not seem easy.
  \end{rem}
  
  \section{Markoff triples and low-lying points in \texorpdfstring{$\L_p$}{ℒₚ}}
  \label{sec:Markoff}
  Before presenting our data on the lowest point in $\L_p$, we must explain certain points in $\L$ that recur in many $\L_p$.
  
  As is classically known (see \cite[Chapter 2, Theorems II and III]{CasselsDio}), the points $\alpha \in \L$ (equivalently $\alpha \in \M$) in the discrete portion below $3$ are parametrized by \emph{Markoff triples,} positive integer solutions $(x,y,z)$ to the \emph{Markoff equation} $x^2 + y^2 + z^2 = 3xyz$. For each Markoff triple with $x \leq y \leq z$, there is an attached irrational number $\xi$ with approximability $\lambda_z = \sqrt{9 - 4/z^2}$. The famous longstanding \emph{Markoff uniqueness conjecture} states that any positive integer $z$ occurs at most once as the largest element of a Markoff triple; it implies that the numbers $\xi$ with approximability values $\lambda_z < 3$ are uniquely determined by their approximability $\lambda_z$ up to $\GL_2 \ZZ$-transformation. We would like to propose a strengthening of this conjecture:
  \begin{conj}[Strong Markoff Uniqueness Conjecture]
    Each of the badly approximable irrationals $\xi$ corresponding to a Markoff triple $(x,y,z)$ has a different field of definition $\QQ\bigl[\sqrt{9 z^2 - 4}\bigr]$. In other words, if $(x,y,z)$ and $(x',y',z')$ are distinct Markoff triples, then
    \[
    \frac{9 z^2 - 4}{9 z'^2 - 4}
    \]
    is not a square in $\QQ$.
  \end{conj}
  To our knowledge, this conjecture has not appeared before in the literature. It is easily checked numerically, and we have computed that it holds for $z, z' \leq 10^{30}$. This conjecture implies a handy characterization for when one of the Markoff points $\lambda_z \in \L$ appears in one of the $p$-spectra $\L_p$.
  \begin{prop}
    Let $(x,y,z)$ be a Markoff triple; let $\lambda_z = \sqrt{9 - 4/z^2} \in \L$ be the corresponding approximability, and let
    \[
    D = \begin{cases*}
      9 z^2 - 4, & $z$ odd \\
      \dfrac{9 z^2 - 4}{4}, & $z$ even
    \end{cases*}
    \]
    be the discriminant of the associated quadratic form. For a prime $p$, we have:
    \begin{enumerate}[$($a$)$]
      \item If $p$ is the product of two principal primes in the quadratic order $\OO_D$ of discriminant $D$, then $\lambda_p \in \L_p$.
      \item Conversely, if $\lambda_p \in \L_p$ and the Strong Markoff Uniqueness Conjecture holds for $z' \leq z$, then $p$ in the product of two principal primes in $\OO_D$.
    \end{enumerate}
  \end{prop}
  \begin{proof}
    The quadratic form $f(x,y)$ of discriminant $D$ corresponds to an ideal class $[\aa]$ of $\OO_D$, invertible because $f$ is primitive, and also ambiguous (that is, $2$-torsion) \cite[Lemma 9]{CasselsDio}. We have $\lambda(\xi) = \lambda_z$ for only one $\GL_2 \ZZ$-class of real numbers $\xi$, namely those for which $\ZZ\<1,\xi\>$ is a representative of $[\aa]$.
    
    If $p$ is the product of two principal primes, necessarily of the form $\pp \ba \pp$, then $\aa \subset \pp\aa$ give two ideals of class $[\aa]$. We may scale the ideals so that $1$ is a primitive element of both of them and write $\aa = \<1, \xi\>$, $\pp \aa = \<1, p\xi\>$. Then $\lambda_p(\xi) = \max\{\lambda(\xi), \lambda(p\xi)\} = \max\{\lambda_z, \lambda_z\} = \lambda_z$.
    
    Conversely, suppose $\lambda_z \in \L_p$, so there is an element $\xi$ such that
    \[
    \lambda_z = \lambda_p(\xi) = \max\{\lambda(\xi), \lambda(p\xi)\}.
    \]
    By flipping $\xi \mapsto p/\xi$, we may assume that $\lambda(\xi) = \lambda_z$. Then $\lambda(p\xi) \leq \lambda_z < 3$, so $p\xi$ is also an irrational of Markoff type, corresponding to a Markoff triple $(x',y',z')$ with $z' \leq z$. But $\xi$ and $p\xi$ are defined over the same quadratic field. By the Strong Markoff Uniqueness Conjecture, this implies that $(x',y',z') = (x,y,z)$. Then the ideals
    $\aa = \<1,\xi\>$ and $\aa' = \<1, p\xi\>$ are  two representatives $\aa \supset \aa'$ of the ideal class $[\aa]$ corresponding to the Markoff form $f$. We have $\aa/\aa'$ cyclic of order $p$. Since $[\aa]$ is invertible, we may form the quotient ideal $\pp = \aa \aa'^{-1}$, a principal prime ideal of norm $p$. We have $p = \pp \ba \pp$, a product of two principal primes as desired. 
  \end{proof}
  If $\Cl(\OO_D)$ is pure $2$-torsion, then Gauss's genus theory shows that the primes $p$ that split into two principal primes in $\OO_D$ are cut out by congruence conditions mod $D$. For the first six Markoff triples, this holds, and we get the following:
  \begin{cor}~\label{cor:Markoff} The presence of points $\alpha \leq 10\sqrt{26}/17 = 2.999423\ldots$ in $\L_p$ is governed by congruence conditions:
    \begin{enumerate}[(a)]
      \item $(1,1,1)\colon$ $\sqrt{5} \in \L_p$ if and only if $p \equiv 0, \pm 1 \mod 5$.
      \item $(1,1,2)\colon$ $2\sqrt{2} \in \L_p$ if and only if $p \equiv 2, \pm 1 \mod 8$.
      \item $(1,2,5)\colon$ $\sqrt{221}/5 \in \L_p$ if and only if $p$ is a square modulo $13$ and $17$.
      \item $(1,5,13)\colon$ $\sqrt{1517}/13 \in \L_p$ if and only if $p$ is a square modulo $37$ and $41$.
      \item $(2,5,29)\colon$ $\sqrt{7565}/29 \in \L_p$ if and only if $p$ is a square modulo $5$, $17$, and $89$.
      \item $(1,13,34)\colon$ $10\sqrt{26}/17 \in \L_p$ if and only if $p \equiv \pm 1 \mod 8$ is a square modulo $13$.
    \end{enumerate}
  \end{cor}
  
  The first four of these show up frequently as $\min \L_p$. Note that the last two points can never be $\min \L_p$, because the associated congruence conditions imply the presence of an even lower point in $\L_p$: $\sqrt{5} \in \L_p$ and $2\sqrt{2} \in \L_p$ respectively.
  
  For the next few Markoff triples following the first six, the class group of $\OO_D$ is not $2$-torsion and, indeed, is expected to be large owing to the presence of a small unit
  \[
  u = \frac{3 z + \sqrt{9 z^2 - 4}}{2} \in \OO_D^\cross.
  \]
  For fixed $z$, by the Chebotarev density theorem, half of all primes split in $\OO_D$ and their factors are equidistributed in the class group of $\OO_D$. It is then perhaps not surprising that there are a good many $p$ for which $\L_p$ has no point below $3$, as we illustrate below.
  
  \section{Numerical data} \label{sec:data}
  Algorithm \ref{alg:main} allows us to compute the minimal element of $\L_p$ for any $p$, assuming the algorithm terminates. In Table \ref{tab:min_Ln}, we list the lowest value of $\L_p$ for primes $p < 2000$ not covered by Corollary \ref{cor:Markoff}, which covers a proportion $55/64 \approx 86\%$ of primes. In Figure \ref{fig:min_Ln}, we show the same data in graphical format.
  
  \begin{table}
    \begin{tabular}{r>{\raggedright\arraybackslash}p{3in}p{1.6in}l}
      $p$    & $\min \L_p = \min \M_p$      & Continued fraction & Symmetry \\ \hline
      $3   $ & $3.46410161513775 = 2\sqrt{3}$ & $[21]$ & S \\
      $67  $ & $3.67821975514076 = \sqrt{7157}/23$ & $[33211112]$ & S \\
      $163 $ & $3.42607262955615 = 3\sqrt{4853}/61$ & $[22121211]$ & R \\
      $227 $ & $3.03973683071413 = \sqrt{231}/5$ & $[22211111]$ & S \\
      $277 $ & $3.04378880403255 = 13\sqrt{29}/23$ & $[2222111]$ & S \\
      $283 $ & $3.04378880403255 = 13\sqrt{29}/23$ & $[2222111]$ & S \\
      $293 $ & $3.33732764987912 = 2\sqrt{19182}/83$ & $[2212111211]$ & R \\
      $317 $ & $3.20024999023514 = \sqrt{6401}/25$ & $[211211111]$ & S \\
      $347 $ & $3.20578401169006 = \sqrt{8643}/29$ & $[22211211111211]$ & S \\
      $397 $ & $3.11986602498262 = \sqrt{2813}/17$ & $[22221]$ & S \\
      $547 $ & $3.11986602498262 = \sqrt{2813}/17$ & $[22221]$ & S \\
      $557 $ & $3.11986602498262 = \sqrt{2813}/17$ & $[22221]$ & S \\
      $587 $ & $3.17316332974547 = 2\sqrt{2117}/29$ & $[2211211]$ & S \\
      $643 $ & $3.11986602498262 = \sqrt{2813}/17$ & $[22221]$ & S \\
      $653 $ & $3.28110118710167 = \sqrt{689}/8$ & $[22121121]$ & S \\
      $683 $ & $3.11986602498262 = \sqrt{2813}/17$ & $[22221]$ & S \\
      $773 $ & $3.20471655999477 = \sqrt{4522693909}/20985$ & $[221121^82112211122111]$ & S \\
      $827 $ & $3.22597102376446 = $ & $[22221121111$ & R \\
      & \qquad $2\sqrt{18981183427599}/2701041$ & \qquad $\allowbreak 211112211111211]$ \\
      $853 $ & $3.11735557792563 = \sqrt{1625621}/409$ & $[2222122111]$ & R \\
      $907 $ & $3.28110118710167 = \sqrt{689}/8$ & $[22121121]$ & S \\
      $947 $ & $3.32456094061007 = 2\sqrt{224439}/285$ & $[22121111121111]$ & A \\
      $997 $ & $3.04963643415425 = \sqrt{174557}/137$ & $[22211222111]$ & S \\
      $1013$ & $3.40942239592361 = 2\sqrt{64517}/149$ & $[221212211]$ & S \\
      $1093$ & $3.00461791388810 = \sqrt{29870597}/1819$ & $[2 2 1 1 2 2 1^{12}]$ & S \\
      $1123$ & $3.08622198700304 = \sqrt{574738353221}/245645$ & $[2 2 2 2 2 2 1 1 1 2 2 1 2 2 1 1 1 1 1 1]$ & R \\
      $1163$ & $3.19289664252119 = \sqrt{9797}/31$ & $[21121111]$ & S \\
      $1213$ & $3.03973683071413 = \sqrt{231}/5$ & $[22211111]$ & S\\
      $1237$ & $3.08977153564575 = \sqrt{1535117}/401$ & $[2 2 1 2 2 1 2 2 1 1]$ & S \\
      $1493$ & $3.33732764987912 = 2\sqrt{19182}/83$ & $[2212111211]$ & A \\
      $1523$ & $3.32267143334034 = \sqrt{10538139}/977$ & $[1 2 1^7 2 1 2 2 1 1 1 1 1]$ & R \\
      $1597$ & $3.16759838060491 = {}$ & $[2^{16}
      1122112112112211]$ & S \\
      & \qquad $\sqrt{11527532430881}/1071860$ \\
      $1627$ & $3.28696105669264 = \sqrt{354606557}/5729$ & $[2 2 2 2 2 1 1 1 2 1 1 2 1 2]$ & R \\
      $1637$ & $3.04378880403255 = 13\sqrt{29}/23$ & $[2222111]$ & S \\
      $1693$ & $3.19289664252119 = \sqrt{9797}/31$ & $[21121111]$ & S \\
      $1747$ & $3.08627411116922 = \sqrt{1720469}/425$ & $[221221^9]$ & S \\
      $1787$ & $3.00487903864412 = \sqrt{60713}/82$ & $[2 2 1 1 2 2 1^7]$ & S \\
      $1867$ & $3.19289664252119 = \sqrt{9797}/31$ & $[21121111]$ & S \\
      $1907$ & $3.04378880403255 = 13\sqrt{29}/23$ & $[2222111]$ & S \\
      $1933$ & $3.25087478718599 = \sqrt{5834157}/743$ & $[211121^{13}]$ & S \\
      $1987$ & $3.00022037758055 = \sqrt{27229}/55$ & $[221^7]$ & S \\
      $1997$ & $3.12428276843314 = 2\sqrt{69697}/169$ & $[222122111]$ & R
    \end{tabular}
    \label{tab:min_Ln}
    \caption{Minimum point of the $p$-Lagrange spectrum $\L_p$ for primes $p < 2000$ not covered by Corollary \ref{cor:Markoff}. For brevity, the terms $1$, $2$, $3$ of the continued fractions are written without separators, and the notation $a^b$ means that the term $a$ is repeated $b$ times.}
  \end{table}

  \begin{figure}[hbtp]
    \centering
    \begin{tikzpicture}
      \begin{axis}[
        width=\textwidth,
        xlabel={Prime $p$},
        ylabel={Lowest point $\alpha = \min\L_p = \min\M_p$},
        legend entries={%
          {Markoff (Cor.~\ref{cor:Markoff})},
          {symmetric non-Markoff (S)},
          {asymmetric alike (A)},
          {reversals (R)}
        }
        ]
        \addplot[
          scatter, only marks,
          point meta=explicit symbolic,
          scatter/classes={
            M={mark=+,mark size=2,green!70!brown},
            S={mark=diamond,mark size=3,brown},
            A={mark=square,mark size=2,blue},
            R={mark=x,mark size=3,red}
          },
        ] table [x=p,y=min, meta=sym] {lowest.dat};
      \end{axis}
    \end{tikzpicture}
    \caption{Minimum point of the $p$-Lagrange spectrum $\L_p$ for primes $p $}
    \label{fig:min_Ln}
  \end{figure}
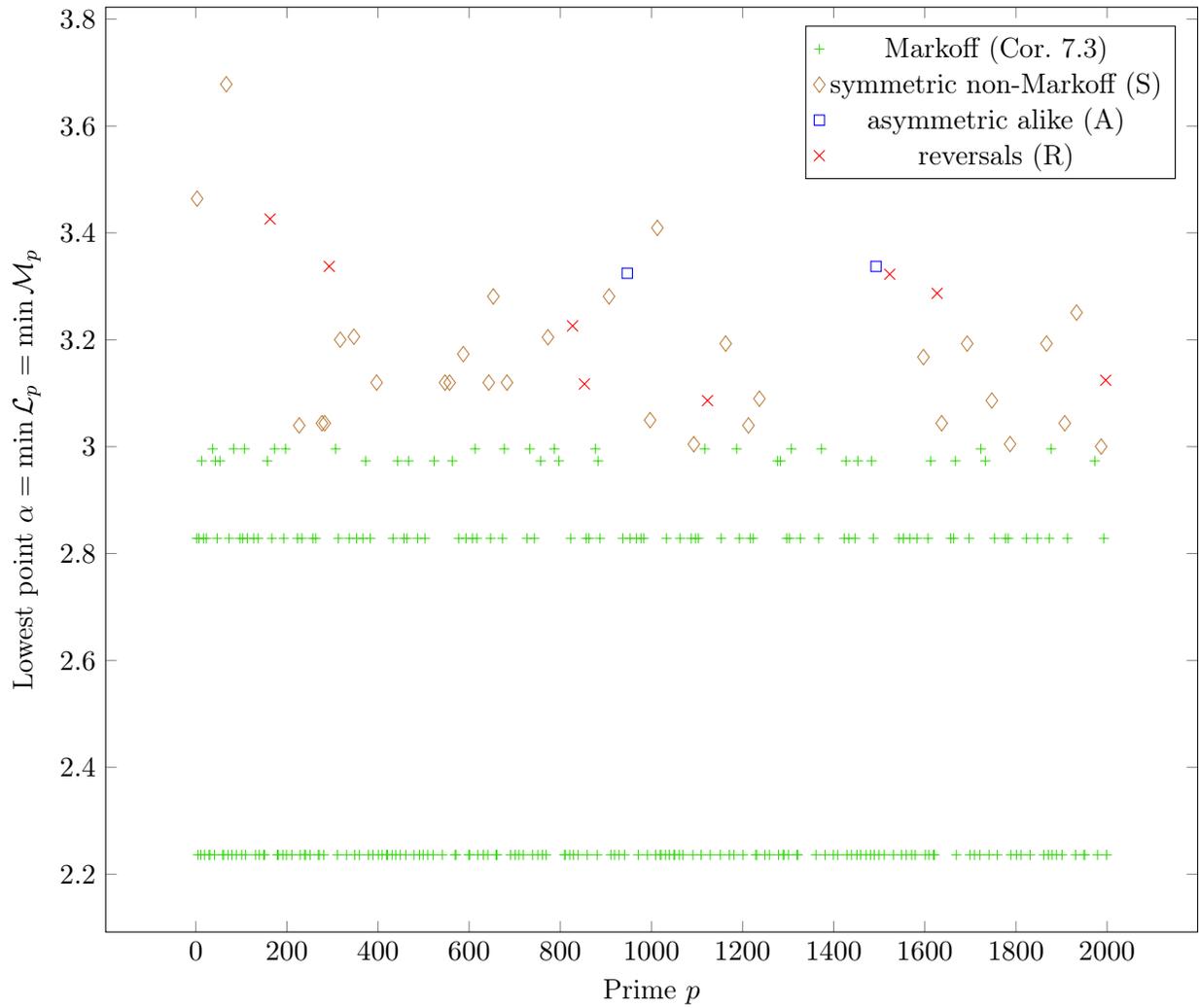
  
  In each case, the continued fraction for $\xi$ is shown (periodic part only). The continued fraction for $p\xi$ is derived by one of three symmetry relations, listed in the last column:
  \begin{itemize}
    \item The continued fraction for $\xi$ and $p\xi$ are \textbf{S}ymmetric and alike (as occurs in all the Markoff cases).
    \item The continued fractions for $\xi$ and $p\xi$ are \textbf{A}symmetric and alike.
    \item the continued fractions for $\xi$ and $p\xi$ are asymmetric and mutual \textbf{R}eversals.
  \end{itemize}
  
  We notice a few patterns. All the values shown are greater than $3$, some only slightly so; that is to say, none of the countably many Markoff numbers show up in $\L_p$. Markoff numbers beyond those covered by Corollary \ref{cor:Markoff} \emph{can} show up as $\min \L_p$, but rarely; for instance, the Markoff triple $(5,29,433)$ gives the minimum point for one out of every $512$ primes, such as
  \[
    \min \L_{3907} = 2.99999644423373236\ldots = \sqrt{9 - \frac{4}{433^2}}.
  \]

  The three largest observed values of $\min \L_p$ occur for $p = 3, 67, 163$, which are (absolute) discriminants of imaginary quadratic fields of class number $1$ (a.k.a.\ Heegner numbers). This does not seem to be entirely a coincidence, although its full significance remains unclear. In both problems the prime $p$ must satisfy the condition that certain small primes are non-squares modulo $p$. The most striking behavior occurs for $p = 67$, which is the unique $p$ found for which $\min \L_p$ lies to the right of the gap $(2\sqrt{3}, \sqrt{13})$ in $\L$ separating continued fractions built of $1$'s and $2$'s from those containing a term at least $3$. In view of the data, and by analogy with the finitude of the Heegner numbers, we conjecture:
  \begin{conj}
    $\min \L_p$ attains a maximum at $p = 67$, where $\min \L_{67} = \sqrt{7157}/23$.
  \end{conj}

  The symmetry relations for $\xi$ and $n\xi$ are striking. Indeed, the S type occurs in a large majority of cases, even for some very long continued fractions where the symmetry cannot be attributed to mere chance. When the continued fraction for $\xi$ is asymmetric, the R relation is apparently more frequent than the A relation, but both occur. Note that the continued fraction $[2,2,1,2,1,1,1,2,1,1]$, with approximability value $\lambda = 2\sqrt{19182}/83$, occurs twice, once for $p = 293$ with the R relation, and once for $p = 1493$ with the A relation. We make the following conjecture (which may be hard, perhaps as hard as the Markoff uniqueness conjecture):

  \begin{conj}
    For each prime $p$, the periodic part of the continued fraction of $\xi$ attaining the minimal $\lambda_p(\xi) = \alpha$ is unique up to reversal. We have $\lambda(\xi) = \lambda(p\xi) = \alpha$, and the periodic parts of $\xi$ and $p\xi$ are either equal or mutual reversals, depending only on $p$.
  \end{conj}
    
\bibliography{biblio}
  \bibliographystyle{plain}
  
\end{document}